\documentclass[11pt]{amsart}

\makeatletter
\let\@wraptoccontribs\wraptoccontribs
\makeatother

\usepackage{amsmath}
\usepackage{amssymb}
\usepackage{amsthm}
\usepackage{amsfonts}
\usepackage{mathrsfs}
\usepackage{latexsym}
\usepackage{hyperref} 
\usepackage{enumerate}
\usepackage[all]{xy}  
\usepackage{color}
\usepackage{comment}

\usepackage{amssymb,amsmath,amscd,url,geometry,pdflscape}
\usepackage{verbatim,url}

\usepackage{marginnote}
\usepackage{comment} 
\numberwithin{equation}{section}

\newtheorem{thm}[equation]{Theorem}
\newtheorem{cor}[equation]{Corollary}
\newtheorem{lem}[equation]{Lemma}
\newtheorem{prop}[equation]{Proposition}
\newtheorem{defn}[equation]{Definition}

\newtheorem{rem}[equation]{Remark}

\def\LLC{\mathrm{LLC}}
\def\HII{\mathrm{HII}}

\def\bd{\mathrm{bd}}

\def\Irr{\mathrm{Irr}}

\def\cI{\mathcal{I}}
\def\cP{\mathcal{P}}

\def\cW{\mathcal{W}}

\def\fo{\mathfrak{o}}
\def\fp{\mathfrak{p}}

\def\fR{\mathfrak{R}}
\def\fe{\mathfrak{e}}
\def\fw{\mathfrak{w}}

\def\Ad{\mathrm{Ad}}
\def\ad{\mathrm{ad}}
\def\gen{\mathrm{g}}
\def\sr{\mathrm{sr}}
\def\sc{\mathrm{sc}}
\def\card{\mathrm{card}}

\def\rH{\mathrm{H}}
\def\rO{\mathrm{O}}
\def\rZ{\mathrm{Z}}
\def\Ru{{\mathrm{R}_{\mathrm u}}}

\def\dim{\mathrm{dim}}
\def\Gal{\mathrm{Gal}}
\def\Hom{\mathrm{Hom}}
\def\H{\mathrm{H}}

\def\Ind{\mathrm{Ind}}
\def\rI{\mathrm{I}}

\def\Nor{\mathrm{N}}

\def\dep{\mathrm{dep}}

\def\sep{\mathrm{s}}

\def\ur{\mathrm{ur}}

\def\enh{\mathrm{e}}
\def\fdeg{\mathrm{fdeg}}

\def\SO{\mathrm{SO}}

\newcommand{\isom}{\xrightarrow{\;\sim\;}}
\def\PGL{\mathrm{PGL}}

\def\GL{\mathrm{GL}}
\def\Aut{\mathrm{Aut}}

\def\R{\mathbb{R}}
\def\C{\mathbb{C}}

\def\Q{\mathbb{Q}}
\def\SL{\mathrm{SL}}
\def\Sh{\mathrm{Sh}}
\def\Int{\mathrm{Int}}
\def\Out{\mathrm{Out}}
\def\Id{\mathrm{Id}}
\def\diag{\mathrm{diag}}

\def\max{\mathrm{max}}

\def\sep{\mathrm{sep}}

\def\j{\mathbf{j}}

\def\h{\mathbf{h}}
\def\g{\mathbf{g}}
\def\M{\mathbf{M}}

\def\U{\mathbf{U}}

\def\tG{\widetilde{{G}}}
\def\bbG{\mathbb{G}}
\def\G{\mathbf{G}}
\def\M{\mathbf{M}}

\def\bG{\mathbf{G}}

\def\bZ{\mathbf{Z}}

\def\As{\mathrm{As}}
\def\rmA{\mathrm{A}}
\def\Stab{\mathrm{Stab}}

\def\W{\mathbf{W}}
\def\fo{\mathfrak{o}}
\def\fp{\mathfrak{p}}
\def\fR{\mathfrak{R}}

\def\TGur{\mathbf T^{\Gur}}
\def\TMur{\mathbf T^{\Mur}}
\def\Mur{\M}
\def\Gur{\G}
\def\SGur{\mathbf S^{\Gur}}
\def\SMur{\mathbf S^{\Mur}}
\def\UGur{\mathbf U^{\Gur}}
\def\UMur{\mathbf U^{\Mur}}
\def\xMur{x^{\Mur}}
\def\tM{\widetilde{M}}

\newcommand{\wt}{\widetilde}
\newcommand{\sC}{\mathscr{C}}
\newcommand{\sB}{\mathscr{B}}
\newcommand{\sA}{\mathscr{A}}

\newcommand{\sM}{\mathscr{M}}
\newcommand{\sR}{\mathscr{R}}
\newcommand{\sS}{\mathscr{S}}

\def\bij{{\mathscr{L}}}
\newcommand{\bR}{\mathbb{R}}
\def\a{\mathbf{a}}

\def\j{\mathbf{j}}

\def\h{\mathbf{h}}
\def\g{\mathbf{g}}
\def\a{\mathbf{a}}

\def\j{\mathbf{j}}

\def\h{\mathbf{h}}
\def\g{\mathbf{g}}
\def\u{\mathbf{u}}

\newcommand{\ra}{\rightarrow}

\newcommand{\<}{\left\langle}
\renewcommand{\>}{\right\rangle}

\begin{document}

\title[Local Langlands correspondence and Weil-restricted groups]{Comparison of the depths on both sides of the local Langlands correspondence 
for Weil-restricted groups}
%{Depth-change under the local Langlands correspondence for Weil-restricted groups}

\author[A.-M. Aubert]{Anne-Marie Aubert}
\address{Institut de Math\'ematiques de Jussieu -- Paris Rive Gauche, CNRS, Sorbonne Universit\'e, Universit\'e de Paris, 
F-75005 Paris, France}
\email{anne-marie.aubert@imj-prg.fr}
\contrib[with an appendix by]{Jessica Fintzen} 
\address{Jessica Fintzen: Trinity College, Cambridge, CB2 1TQ, UK} \email{fintzen@umich.edu}
\author[R. Plymen]{Roger Plymen}
\address
%{School of Mathematics, Southampton University, Southampton SO17 1BJ,  England \emph{and} 
{School of Mathematics, Alan Turing Building, Manchester University, Manchester M13 9PL, England}
\email{roger.j.plymen@manchester.ac.uk}

%\classno{20G05, 22E50}
\date{\today}
\maketitle

\begin{abstract}  
Let $E/F$ be a  finite and Galois extension of non-archimedean local fields. Let $G$ be a connected reductive group defined over $E$ and let $M: = \fR_{E/F}\, G$  be the reductive group over $F$ obtained by Weil restriction of scalars. 
We investigate  depth, and the enhanced local Langlands correspondence, in the transition from $G(E)$ to $M(F)$.   We obtain a depth-comparison  formula for Weil-restricted groups.  
\end{abstract}

\tableofcontents

\section{Introduction}   
Let $E/F$ be a  finite Galois extension of non-archimedean local fields.    
Let $G$ be a connected reductive group defined over $E$ and let $M: = \fR_{E/F}\, G$  be the reductive group over $F$ obtained by Weil restriction of scalars from $G$.   
We have an isomorphism of locally compact totally disconnected topological groups $\iota\colon G(E)\ra M(F)$ between the $E$-points of $G$ and the $F$-points of $M$.
We investigate the depth, and the enhanced local Langlands correspondence, in the transition from $G(E)$ to $M(F)$.

We denote by $\Pi(G,E)$ the smooth dual of $G(E)$, the set of equivalence classes of irreducible smooth representations of $G(E)$.   We denote by $\Phi(G,E)$ the set of $G^\vee$-conjugacy classes of Langlands parameters for $G(E)$, where $G^\vee$ is the complex dual group of $G$.   Similarly for $\Pi(M,F)$ and $\Phi(M,F)$.  

%Write ${}^LG:= G^\vee \rtimes W_E$, and let $\bHom(W_E', {}^LG)$ denote the set of homomorphisms $\phi \colon W'_E\to{}^LG$ which are admissible as defined  by \cite[8.2]{Bor}. Let $\Phi(G,E)$ be the set of equivalence classes %of elements in $\bHom(W_E', G^\vee \rtimes W_E)$ modulo  inner automorphism by elements of $G^\vee$. Define
%$\Phi(M,F)$  in the analogous way.  Let $\Pi(G(E))$ denote the smooth dual of $G(E)$, the set of equivalence classes of irreducible smooth representations of $G(E)$.   Likewise for $\Pi(M(E))$.

We assume that $G(E)$ admits a local Langlands correspondence, i.e. a surjective map
\[
\lambda_G \colon \Pi(G,E) \to \Phi(G,E)
\]
 which satisfies the conditions (desiderata) laid down in \cite[\S 10]{Bor}.   
The map $\lambda_M$ is defined to be the unique  map  for which the following diagram commutes
\[
\begin{CD} \Pi(M,F)@ > \lambda_M >> \Phi(M,F)\\
@V \iota^* VV                       @VV \Sh V\\
\Pi(G,E)@> \lambda_G >> \Phi(G,E)
\end{CD}
\]
where  $\Sh$ is the restriction to $\Phi(M,F)$ of the Shapiro isomorphism, as in \cite[8.4]{Bor}.

Each side of the correspondence $\lambda_G$ admits a numerical invariant, namely the \emph{depth}.   
\begin{comment}
A given parameter $\phi \in \Phi(G,E)$ determines a cocycle $a_{\phi} \in \H^1(W_E, G^\vee)$.   
A given smooth irreducible representation $(\pi, V_{\pi}) \in \Pi(G,E)$ also has a depth, which comes about as follows.     For each point $x$ in the reduced Bruhah--Tits building 
$\sB(G,E)$, we have the isotropy subgroup $G(E)_x$ and the Moy--Prasad filtration 
\[
G(E)_x = G(E)_{x,0} \rhd G(E)_{x, r_1}  \rhd  G(E)_{x,r_2}  \rhd  \cdots
\]
with $0 < r_1 < r_2 < \cdots$.      Informally, we move along the parameters $r_j$ until we find that $V$ admits nonzero fixed vectors: 
\[
V_{\pi}^{G(E)_{x,r_j}} \neq 0.   
\]
The real number $r_j$ obtained in this way is the depth $\dep(\pi)$.   For the precise definition, see \S 4. 
\end{comment}
The depth is defined in quite different ways on each side of the LLC (local Langlands correspondence), and, \emph{a priori}, there is no reason to expect that this numerical invariant will be preserved by the LLC.   
However, for $\GL_n$ and its inner forms, we do have preservation of depth, see \cite{ABPS}.  

In large residual characteristic, we also have the depth-preserving property of the LLC  for quasi-split classical groups 
and  for arbitrary unitary groups, \cite{Oi1}, \cite{Oi2}.   For tamely ramified tori, we have depth-preservation under the LLC, see Yu \cite{Yu1}.   

For inner forms  of $\SL_n$, we have depth-preservation under the LLC for essentially tame Langlands parameters \cite[Theorem 3.8]{ABPS}.

Mark Reeder discovered that, for a certain classical octahedral representation of $\SL_2(\Q_2)$, depth is not preserved by the LLC, see \cite[Example 3.5]{ABPS}.     

It is in this spirit that, in this article, we investigate  depth, and the local Langlands correspondence, in the transition from $G(E)$ to $M(F)$.

Returning to our commutative diagram, the map $\lambda_M$ is the composition of three maps, and there are therefore three separate opportunities for a change of depth. 
 In this article, for each of these maps, we record how the depth can change.   

In the bottom horizontal map, the depth change will vary from group to group: for $\GL_n$ we have preservation of depth \cite{ABPS}.   

In section~\ref{sec:Shapiro}, we prove  a depth-comparison result (Theorem \ref{depsh}) for the right vertical map $\Sh$.     The depth change will depend on the classical Hasse-Herbrand function $\varphi_{E/F}$.   

The Appendix, due to Jessica Fintzen,  is devoted to a depth-comparison result (Corollary \ref{ethm}) for the left vertical map $\iota^\ast$. The depth change will depend on the ramification index $e = e(E/F)$.

In  section~\ref{sec:LLC}, we recall the classical desiderata of Borel \cite{Bor} and Langlands \cite{Lan} for a local Langlands correspondence.   

An \textit{enhanced $\LLC$} $\lambda_G^\enh$ includes additional (refined) data that pin down the smooth irreducible representations of the $p$-adic group, i.e. divide each 
$L$-packet into singletons. It is a bijection 
\[\begin{matrix}\lambda_G^\enh\colon&\Pi(G,E)&\to&\Phi^\enh(G,E)\cr
&\pi&\mapsto&(\phi_\pi,\rho_\pi)
\end{matrix},
\]
such that  $\phi_\pi=\lambda_G(\pi)$, and where $\Phi^\enh(G,E)$ is the set of $G^\vee$-conjugacy classes of enhanced $L$-parameters (see Definition~\ref{defn:enh}), which 
satisfies several stringent conditions.  These extra conditions, which are  made precise  in Definition~\ref{defn:refined LLC},  comprise  Whittaker data,  the $\HII$ conjecture for square-integrable modulo center representations, extended endoscopic triples, transfer properties.

%From a Whittaker datum $\fw$ for $G(E)$, we construct a Whittaker datum $\fw_{E/F}$ for the Weil-restricted group $M(F)$.   
%\begin{thm}  If  $\phi \in \Phi(M,F)$ then we have \[\dep(\Sh(\phi)) = \psi_{E/F}(\dep(\phi))\] where $\psi_{E/F}$ is the inverse of the Hasse-Herbrand function.     \end{thm}
%In section~\ref{sec:LLC}, we provide a detailed description of the map \[\lambda_M = (\Sh^{-1}) \circ \lambda_G \circ \iota^*.\]
%\begin{thm}   If $\pi \in \Pi(M,F)$ then we have \[\dep(\iota^* \pi) = e(E/F)  \cdot \dep(\pi)\]where $e = e(E/F)$  the ramification index of $E/F$.   \end{thm}

This leads to our main result in section~\ref{sec:LLC}.

\begin{thm} \label{intro 1}  Consider a local Langlands correspondence 
\[
\lambda_G\colon\Pi(G,E) \longrightarrow \Phi(G,E). 
\]
Then

\begin{enumerate}
\item[{\rm (i)}]   the map $\lambda_M$ defined above induces an enhanced $\LLC$ for $(M,F)$ if and only if $\lambda_G$ induces an enhanced $\LLC$ for $(G,E)$.
\item[{\rm (ii)}] if $\pi \in \Pi(M,F)$  and $\dep(\lambda_G(\iota^* \pi))  =  \kappa_\pi \cdot \dep(\iota^* \pi)$ then we have 
\[
 \dep(\lambda_M(\pi))  = \varphi_{E/F}(\kappa_\pi \cdot e \cdot \dep(\pi)),
 \]
where $e = e(E/F)$ is the ramification index of $E/F$ and $\varphi_{E/F}$ is the Hasse-Herbrand function. 
\end{enumerate}
\end{thm}

%\begin{thm}\label{intro1} Consider the local Langlands correspondence \[\lambda_G\colon \Pi(G,E) \longrightarrow \Phi(G,E). \]
 %  If $\pi \in \Pi(M(F))$ and $\dep(\lambda_G(\iota^* \pi))  =  \kappa_\pi \cdot \dep(\iota^* \pi)$ then we have \[ \dep(\lambda_M(\pi))  = \varphi_{E/F}(\kappa_\pi \cdot e \cdot \dep(\pi)), \]
%where $\varphi_{E/F}$ is the Hasse-Herbrand function. \end{thm} 

As a special case, we have 

\begin{thm} \label{intro2}  If $\lambda_G$ is depth-preserving, then, for all $\pi \in \Pi(M,F)$,  we have
\[
\dep (\lambda_M (\pi)) =  \varphi_{E/F}(e \cdot \dep( \pi)).
\]   

In particular, we have
\begin{itemize}
\item $\dep(\lambda_M(\pi)) / \dep(\pi) \to 1 \quad \textrm{as} \quad \dep(\pi) \to \infty$ 
 \smallskip 
\item $\lambda_M$ is depth-preserving if and only if $E/F$ is tamely ramified,
\smallskip
\item if $E/F$ is wildly ramified then, for each  $\pi$ with $\dep(\pi)>0$,
 we have
\[
\dep(\lambda_M(\pi) ) > \dep(\pi).
\]
\end{itemize}
\end{thm}

When $G(E) = \GL_1(E)$, Theorem \ref{intro2} strengthens the  main result of \cite{MiPa} for induced tori.   For tamely ramified induced tori, we recover the depth-preservation theorem  of Yu \cite{Yu1}.    

Section~\ref{sec:App} contains some applications to automorphic induction and the Asai lift.

Let $E$ be a local field of residual characteristic $2$, let $G(E) = \SL_2(E)$.   Then $\lambda_G$ is depth-preserving if and only if $\phi$ is essentially tame \cite{AMPS}.  
Examples when $\phi$ is not essentially tame are provided in \cite{AMPS}.   They include the classical octahehedral representations, the classical tetrahedral representations, and the case when a lift of $\phi$ with minimal depth among the lifts of $\phi$ is imprimitive and totally ramified.   These examples  
all satisfy $\kappa_\pi <1$, in agreement with \cite[1.1]{AMPS}.   

%Let $\varphi$ denote a  lift of $\phi$ with minimal depth among the lifts of $\phi$.\begin{itemize}
%\item The original example of Mark Reeder for $\SL_2(\Q_2)$.   Here, the Langlands parameter $\varphi\colon W_{\Q_2} \to \GL_2(\C)$  is trivial on $\SL_2(\C)$ and has  
%image isomorphic to the symmetric group $\mathfrak{S}_4$.   We have \[\kappa = \frac{2}{3}\] see \cite[Example 3.5]{ABPS}.
%\item The classical octahedral representations.  Here, $\varphi$ is octahedral, that is, it is primitive and $\phi(W_E) = \mathfrak{S}_4$ the symmetric group.   We have 
%\[\kappa_\pi = \frac{4r}{5r+1}\] with $r$ a certain ramification depth, $r-1 \in 6\Z$, see \cite[\S 2]{AMPS}. The example with  $r=1$ and $E = \Q_2$ was discovered by Mark Reeder \cite[Example 3.5]{ABPS}.   
 %we obtain the previous example.
 %original example of Mark Reeder \cite[Example 3.5]{ABPS}  for which $\kappa = 2/3$.    
%\item The classical tetrahedral representations.   Here, $\varphi$ is tetrahedral, that is, it is primitive and $\phi(W_E) = \mathfrak{A}_4$ the alternating group.   We have
%\[\kappa_\pi = \frac{4r}{5r+1}\] with $r$ a certain ramification depth, $r-1 \in 6\Z$, see \cite[\S 2]{AMPS}.
% \item Here,  $\varphi$ is imprimitive and totally ramified.   Then there exists a separable totally ramified quadratic extension $L/E$ and a 
 %character $\xi$ of $W_L$ such that $\varphi = \ind_{W_L}^{W_E} \xi$.   We have  \[ \kappa_\pi = \frac{\dep(\xi^2) + d(L/E)}{\dep(\xi) + d(L/E)} < 1 \]
 %where $\mathfrak{p}_E^{1 + d(L/E)}$ is the relative discriminant of $L/E$, see \cite[\S 2]{AMPS}.   \end{itemize}

We would like to thank  Masao Oi for several valuable emails which enabled us to improve sections \ref{sec:Shapiro} and \ref{sec:App}.        We would also like to express our gratitude to Jessica Fintzen for providing the appendix.

\textbf{Notation and conventions.}

Let $E/F$ be a  finite Galois extension of non-archimedean local fields, with ramification index denoted by $e(E|F)$.
We fix a separable closure $F^{\sep}=E^{\sep}$ of both $F$ and $E$. From now on all field extensions will be assumed to be contained in $F^{\sep}$. 
Let $F^{\ur}$ denote the maximal unramified extension of $F$.

Let $K$ be any non-archimedean local field contained in $F^{\sep}$.  Let $\Gamma_K$ be the absolute Galois group of $K$, let $W_K$ be the absolute Weil group of $K$, and let $W'_K$ denote the Weil-Deligne group $W_K\times\SL_2(\C)$. 

Let $H$ be a connected reductive group defined over $K$ and let $\Pi(H,K)$ denote the set of isomorphism classes of irreducible admissible representations of $H(K)$.
Let $H^\vee$ be the Langlands dual of $H$.   
Write ${}^LH:= H^\vee \rtimes W_K$.  Homomorphisms $\phi \colon W_K'\to{}^LK$ which are admissible as defined in \cite[8.2]{Bor}
are called $L$-parameters for $(K,K)$. We denote by $\Phi(H,K)$  the set of $H^\vee$-conjugacy classes of  (resp. bounded) $L$-parameters for $(H,K)$.

Let $\sB(H,K)$ be the reduced Bruhat--Tits building of $H$ over $K$. We denote by  $H(K)_{x,r}$ the Moy--Prasad filtration, where $x\in\sB(H,K)$ and $r\in\R_{\ge 0}$, as defined in \cite{MoPr1}, \cite{MoPr2}, and write $H(K)_{x,r+}:=\bigcup_{s>r} H(K)_{x,s}$.
\begin{defn} \label{defn:depth}
Let $(\pi,V_\pi)$ be an irreducible smooth representation of $H(K)$.
The depth $\dep(\pi)$ of $\pi$ the smallest nonnegative real number $r$ for which there exists an $x \in \sB(H,K)$ such that $V_\pi^{H(K)_{x,r}+}\ne 0$. 
\end{defn}
%\mar{\textcolor{blue}{AM: I refered to \cite{De} in order to have a definition of the depth that is independent of the Lie algebra. As noticed, on page 242 of [De], it seems to be necessary because the usual definition \cite{MoPr2} requires that $\fg(E)_{x,r}/\fg(E)_{x,r}\simeq G(E)_{x,r}/G(E)_{x,r+}$ for $r > 0$, which is not always true as observed by Yu in \cite[\S 0.4]{Yu2}.}}
The definition of depth given in Definition~\ref{defn:depth} makes sense, see \cite[Lemma~5.2.1]{De}.

Let $G$ be a connected reductive algebraic group defined over $E$ and let $M:= \fR_{E/F}\, G$  be the reductive group over $F$ obtained by Weil restriction of scalars from $G$.  Let $G(E)$ denote the group of $E$-points of $G$, and $M(F)$ the $F$-points of $M$. 
We will denote by $\Phi(G,T)$ the root system of $G_{E^{\sep}}$ with respect to $T_{E^{\sep}}$ for $G$ some group defined over some subfield of $E^{\sep}$ and $T$ a (not necessarily maximal) torus in $G$. Let $X^*(T)$ denote the group of algebraic characters of $T$.  

The set $M(F)$ has the structure of  a variety defined over $F$, and  the set $G(E)$ has the structure of a variety defined over $E$.   The sets $M(F)$ and $G(E)$ are in bijective correspondence.   They are homeomorphic as topological spaces.   
 Once the group structure on $G(E)$ is transported to 
$M(F)$, we have an isomorphism of locally compact totally disconnected topological groups:
\begin{eqnarray}\label{RR1}
\iota\colon G(E)\ra M(F).
\end{eqnarray}
Therefore, $G(E)$ and $M(F)$ have the same representation theory. 
 Let
 \begin{eqnarray}\label{RR2}
 \iota^*\colon \Pi(G,M)\to \Pi(G,E)
 \end{eqnarray}
 be the canonical bijection.

%We will write $A \times_F E$ for $A \times_{\Spec F} \Spec E$.

 %Recall that $M^\vee=I_{W_E'}^{W_F'} (G^\vee)$ in the notation of \cite[\S 4]{Bor}.

 \section{Depth-comparison under the Shapiro isomorphism}     \label{sec:Shapiro}
Let $\cI_F$ be the inertia subgroup of $W_F$ and $\cP_F$ the wild inertia subgroup.  
Attached to a real number $r \geq -1$ is the ramification subgroup 
$W_F^r$ of $W_F$.   We use the upper numbering convention of \cite[Chap. IV]{Ser1}, so that $W_F^{-1} = W_F$.  We have the semi-continuity property
\[
W_F^r = \bigcap_{s<r} W_F^s
\]
for all $r>0$.   

One also forms the subgroup 
$\bigcup_{s>r} W_F^s$ and its closure 
\[
W_F^{r+} = \textrm{cl}(\bigcup_{s>r} W_F^s)
\]
in $W_F$.   
   One says that $r$ is a \emph{jump} of $\overline{F}/F$ if $W_F^{r+} \neq W_F^r$.   In particular,
\[
W_F^0 = \cI_F, \quad \quad W_F^{0+} = \cP_F.
\]
Each of the groups $W_F^r, W_F^{r+}$ with $r \geq 0$ is profinite, closed and normal in $W_F$.   If $r>0$ then $r$ is a jump if and only if $r \in \Q$, see
\cite[\S 2.4, Corollary]{BH}.

%We set $\cI_F^r = W_F^r$ and $\cI_F^{r+} = W_F^{r+}$  for $r\geq 0$, so that $\cI_F^0 = \cI_F$ and $\cI_F^{0+} = \cP_F$.   

 \textsc{The inverse $\psi_{E/F}$ of the Hasse--Herbrand function \cite[IV \S3]{Ser1}.}  Let $e = e(E/F)$.   Let $\Gamma$ denote the Galois group $\Gal(E/F)$ and consider the upper numbering of the ramification groups.
Let $j$ be the largest jump in the upper numbering of the ramification groups of $\Gamma$.     With $x \geq j$ we have 
\begin{eqnarray*}
\psi_{E/F}(x) & = & \int_0^x(\Gamma^0 : \Gamma^w) dw \\
& = & \psi_{E/F}(j) + \int_{j}^x(\Gamma^0: 1) dw \\
& = & \psi_{E/F}(j) + (x - j) e\\
& = & \psi_{E/F}(j) - je +ex
\end{eqnarray*}
so that $\psi_{E/F}(x)$ with $x \geq j$ is simply a \emph{translate} of the linear function $x \mapsto  ex$.   For sure, we have
\[
\psi_{E/F}(x) \leq e x
\]
with equality if and only if $j = 0$ i.e. if and only if $E/F$ is tamely ramified; 
and 
\[
\psi_{E/F}(x) \sim e x
\]
as $x \to \infty$. 

The following lemma is due to Bushnell-Henniart \cite[\S1.4]{BH}.
   
\begin{lem} \label{comp}  The comparison lemma.  If $r>0$, then   
  \begin{eqnarray*}
W_F^r \cap W_E &=& W_E^{\psi_{E/F}(r)}, \\
W_F^{r+} \cap W_E &=& W_E^{\psi_{E/F}(r)+}.
 \end{eqnarray*}
  \end{lem}

  We require the second statement in Lemma \ref{comp} for $r = 0$.  For this, let $X$ be a topological space, let $A$ be an open and closed subset of $X$. Then,  for every subset $B$ of $X$,
   according to \cite[\S1.6 prop. 5]{Bourbaki}  we have
     \[
     \overline{B} \cap A = \overline{B \cap A}
     \]
  where the overline denotes closure in $X$.    Now let  $X = W_F, A = W_E$, $B = \bigcup_{s>0}W_F^s$.   We obtain 
  \begin{equation}
  W_F^{0+} \cap W_E = W_E^{0+}.
  \end{equation}

\begin{comment}
 \begin{proof}   Let   $L/F $ be a Galois extension  which contains  $E$.  According to standard material, e.g. \cite{Ser1}, we have
  \[
 \Gal(L/F)^r \cap \Gal(L/E)  = \Gal(L/E) ^{\psi_{E/F}(r)}.
 \]
 Take the inverse limit over all such extensions $L$, we obtain
 \[
  \Gamma_F^r \cap \Gamma_E = \Gamma_E^{\psi_{E/F}(r)} .
  \]
  In the inclusion $W_F \subset \Gamma_F$, $W_E$ is defined to be the pre-image of $\Gamma_E$, see \cite[p.3]{Tate}.   
 That is to say, we have $W_F \cap \Gamma_E = W_E$.   Now intersect the above equation with $W_F$.  We obtain
  \[
 W_F^r \cap W_E = W_E^{\psi_{E/F}(r)}.
 \] 
 Then the second equality follows.
 \end{proof} 
 \end{comment}

\begin{rem} {\rm The equalities in Lemma~\ref{comp} are still valid for any finite separable extension $E/F$, see \cite[Proposition~1.4]{BH}. Also, as already noticed in \cite[Lemma~1]{BH1}, if $E/F$ is a finite, tamely ramified field extension with $e = e(E|F)$, then $\cP_E = \cP_F$ and 
\begin{eqnarray*}
&W_F^r = W_E^{er}, & \text{if  $r > 0$,}\\
&W_F^{r+} = W_E^{er+}, & \text{if $r \geq 0$.}
\end{eqnarray*}
}
\end{rem}

 \begin{comment}
 \textsc{Group extensions}.\label{ext}   Let $\g$ be a profinite group, let $\j$ be a closed normal subgroup in $\g$, and let $\a$ be a $\g$-group.   The group $\g/\j$ acts on $\a^\j$, 
 which means that $\H^1(\g/\j, \a^\j)$ is well-defined.    Suppose that $\alpha, \beta \in Z^1(\g,\a)$ are trivial on $\j$ and are cohomologous in $\H^1(\g, \a)$.    Then $\alpha, \beta$ are cohomologous in 
 $Z^1(\g/\j, \a^\j)$ and there is a canonical inclusion 
 \begin{eqnarray}\label{hook}
 \H^1(\g/\j, \a^\j) \hookrightarrow \H^1(\g, \a).
 \end{eqnarray}
 See \cite[\S 1.27]{BS} and \cite[Chapter I, \S 5.8]{Ser2}.  We also have
    \begin{eqnarray}\label{running}
  \H^1(\g, \a) = \varinjlim \H^1(\g/\u, \a^\u)
  \end{eqnarray}
 for $\u$ running over the set of open normal subgroups of $\g$, see \cite[\S 5.1]{Ser2}. 
 \end{comment}

 \textsc{Induction}.    We require the continuous noncommutative cohomology as developed in Borel-Serre \cite{BS}.    So, let $\g$ denote a topological group.   
 A $\g$-set is a discrete topological space $X$ on which $\g$ operates on the left in a continuous fashion (i.e. the isotropy subgroup of each point of $X$ is open in $\g$).     
 A $\g$-group $A$ is a group in the category of $\g$-sets, as in \cite[\S 1.2]{BS}.   The cohomology set $\H^1(\g, A)$ is defined in \cite[\S 1.2]{BS}: 
 it is constructed from continuous cocycles of $\g$ with values in $A$.   Then $\H^1(\g, A)$ is   a pointed set - the distinguished point is the class of the unit cocycle.   
 
 If $\h$ is a subgroup of $\g$ and $A$ is a $\g$-group, then the induced group $A^*$ is defined in \cite[\S 1.28]{BS}:
 \[
 A^*: = \Ind_{\h}^{\g}(A)
 \]
 It comprises all continuous $\h$-equivariant maps from $\g$ to $A$ which are constant on left cosets modulo an open subgroup of $\g$.   Then $A^*$ becomes a $\g$-group via 
 \[
 ({}^gf)(x) = f(xg)
 \]
  for all $f \in A^*$ and $g,x \in \g$.

 \begin{thm}\label{Sh0} \cite[Proposition 1.29]{BS}.   Let $\h$ be a subgroup of $\g$, let $A$ be an $\h$-group, and let $A^*$ be the corresponding induced $\g$-group.   Suppose that the open 
 normal subgroups of  $\g$ form a fundamental system of neighbourhoods of $1 \in \g$.   Then we have a canonical isomorphism of pointed sets:
 \begin{eqnarray}\label{Sh1}
  \H^1(\g, A^*) \simeq \H^1(\h, A).
  \end{eqnarray}
  \end{thm}
 This is the \emph{Shapiro isomorphism}, denoted $\Sh$.   
 
 Here is an important application of the Shapiro isomorphism.  Let $G^\vee$ (resp. $M^\vee$) denote the complex Langlands dual of $G$ (resp. $M$).   
  Let $\g = W_F, \h = W_E, A = G^\vee$.   Then $M^\vee$ is the induced group $\Ind_{W_E}^{W_F}(G^\vee)$.  The Weil groups $W_E$ and $W_F$ are locally profinite topological groups.
  The inertia subgroup $\cI_E$ (resp. $\cI_F$) contains open normal subgroups forming a fundamental system of neighbourhoods of the identity in $W_E$ (resp. $W_F$).  
   
   We will take $G^\vee$ in its discrete topology, so that 
  $G^\vee$ becomes a $W_E$-group;  and $M^\vee$ in its discrete topology, so that $M^\vee$ becomes a $W_F$-group. 
  We can now apply Theorem  \ref{Sh0}, and obtain the canonical isomorphism of pointed sets
 \begin{eqnarray}\label{Sh2}
 \H^1(W_F, M^\vee) \simeq \H^1(W_E, G^\vee).
 \end{eqnarray}
 
 Here is a much more specialized application, which we will need in the proof of Theorem \ref{shr}.   Let
 \[
 \h = W_E/W_E^{\psi(r)+}, \quad \g = W_F/W_F^{r+}, \quad A = (G^\vee)^{W_E^{\psi(r)+}}
 \]
 
 To simplify notation, set $G_1 = W_E$, $G_2 = W_F$, $H_1 = W_E^{\psi(r)+} $, $H_2 = W_F^{r+}$, so that  $G_1 \subset G_2$, $H_1 \subset H_2$, $H_1 = H_2 \cap G_1$ by the comparison lemma \ref{comp}.   
 We need to show that 
 the map 
 \[
\eta \colon  \h \to \g, \quad xH_1 \mapsto xH_2
 \]
  is injective.    To prove this, note that, for all $x \in G_1$ we have  
 \[
  xH_2 = H_2 \implies x \in H_2 \implies x \in H_2 \cap G_1 \implies x \in H_1 \implies xH_1 = H_1.
 \]
 Setting $x = z^{-1}y$ with $y,z \in G_1$ we infer that $yH_2 = zH_2 \implies yH_1 = zH_1$ as required.
  We identify $\h$ with its image $\eta(\h) \subset \g$, and view $\h$ as a subgroup of $\g$.      
 
 We can therefore apply Theorem \ref{Sh0} and obtain the  canonical isomorphism
 \begin{equation}\label{Sh3}
 \H^1\left(W_F/W_F^{r+}, \Ind_{W_E/W_E^{\psi(r)+}}^{W_F/W_F^{r+}}(G^\vee)^{ W_E^{\psi(r)+}}  \right ) \simeq \H^1\left(W_E/W_E^{\psi(r)+}, (G^\vee)^{W_E^{\psi(r)+}}\right).
 \end{equation}

 The following lemma is observed in \cite[Lemma 3]{MiPa}.     
  \begin{lem} \label{ABC} The submodule lemma.     Let $J$ be a group, $H$ and $A$ subgroups of $J$ with $A$ being normal in $J$.  Let $B = H\cap A$, let $M$ be an $H$-module.   Then there is a canonical isomorphism 
  of $J/A$-modules:
 \[
 (\Ind_H^J M)^A \simeq \Ind_{H/B}^{J/A} M^B.
 \]
 \end{lem} 
 
 We shall need this lemma in the proof of Theorem \ref{shr}.   
 
For the present, let $G^\vee$ (resp. $M^\vee$) denote the complex Langlands dual of $G$ (resp. $M$).   
As in \cite[8.2, Eqn.(1)]{Bor}), we have the canonical injective map   
\begin{equation} \label{eqn:PhiH1}
\Phi(M,F) \hookrightarrow \H^1(W'_F,M^\vee)
\end{equation}
which allows us to view $\phi \in \Phi(M(F))$  as a cocycle $a_{\phi} \in Z^1(W'_F, M^\vee)$.  The restriction of this cocycle to $W_F$ will 
 be denoted $\alpha_{\phi}$, and its cohomology class in $\H^1(W_F, M^\vee)$ will be denoted $[\alpha_{\phi}]$.

 \begin{lem}\label{union} We have a canonical isomorphism of pointed sets:
  \begin{equation*} \label{eqn:H1}
 H^1(W_F,M^\vee)=\bigcup_{r\in\R_{\geq 0}} H^1(W_F/W_F^{r+},(M^\vee)^{W_F^{r+}}).
 \end{equation*}
 \end{lem}
  
 \begin{proof}   The group $W_F^{r+}$ is a normal subgroup of $W_F$ for $r \geq 0$.   According to \cite[Proposition 1.27]{BS}, we have 
 a canonical injective map
 \begin{equation}\label{inj}
 \varinjlim_{r \in \R_{\geq 0}} \H^1(W_F/W_F^{r+}, (M^\vee)^{W_F^{r+}}) \to \H^1(W_F, M^\vee).
 \end{equation}
 
  %Since the direct limit commutes with taking kernel and cokernel, it suffices to show that this map is bijective.    The injectivity follows from the definition of inductive limit.   
 We check the surjectivity of this map.   
 
  The ramification filtration $\{W_F^r\}_{r \geq 0}$ is descending and satisfies 
 \begin{equation}\label{intersect}
 \bigcap_{r \geq 0} W_F^r = \{1\}
 \end{equation}
     
The dual group $M^\vee$ is equipped with the discrete topology, and is an $W_F^0$-group in the terminology of \cite{BS}.   That is, 
there is a continuous action of $W_F^0$ on  $M^\vee$, i.e. we have a continuous homomorphism $\rho : W_F^0 \to \Aut(M^\vee)$.   Since $\Aut(M^\vee)$ is also discrete,  
the kernel of $\rho$ is an open normal subgroup $U \subset W_F^0$.   
%Now $W_F^0$ is profinite, therefore $\rho$ factors though a finite quotient of $W_F^0$.   
By (\ref{intersect}), we will have $W_F^{r+} \subset U$  for sufficiently large  $r$,  say $r \geq r_0$.     In other words, given $\alpha \in Z^1(W_F, M^\vee)$, 
the image of $\alpha$ is contained in $M^\vee = (M^\vee)^{W_F^{r+}}$ for   $r \geq r_0$. 
 
By the continuity (smoothness) of $\alpha$, $\alpha$ is trivial on some open subgroup $H$ of $W_F^0$, thus $H$ is of finite index in $W_F^0$.     
Again by (\ref{intersect}), $W_F^{r+}$ is contained in $H$ for sufficiently  large $r$, say $r \geq r_1$.      Therefore, for $r \geq \max(r_0, r_1)$,  $\alpha$ belongs to $Z^1(W_F/W_F^{r+}, (M^\vee)^{W_F^{r+}})$.   

To complete the proof, we combine this data with the injective map (\ref{inj}). 
  
 \end{proof}
 
Lemma \ref{union} allows us to present a new definition of depth.    

%In \cite{MiPa}, the authors essentially worked with the following definition of depth.
   
 \begin{defn} \label{Oidef} For $\phi \in \Phi(M,F)$, we define the depth of $\phi$ as
 \[
 \dep(\phi): =  \inf \{ r\in \R_{\geq 0} : [\alpha_{\phi}] \in \H^1(W_F/W_F^{r^+} ,(M^\vee)^{W_F^{r+}})\} .
 \] 
 \end{defn}

 If $M(F)$ is tamely ramified, then the wild inertia group $\cP_F = W_F^{0+}$ acts trivially on $M^\vee$.   
In particular, we may regard the restriction  $\alpha_{\phi}|_{W_F^{0+}}$ of $\alpha_{\phi}$ to $W_F^{0+}$ as a continuous homomorphism from $W_F^{0+}$ to $M^\vee$.  
%Note that  $\phi(w) = (\alpha_{\phi}(w),w)$ for all $w \in W_F$.   

For tamely ramified groups, the customary definition is as follows.

\begin{defn} \label{deftame}
For tamely ramified groups, the usual depth of $\phi$ is defined as follows:
\[
\dep(\phi) : = \inf\{r \geq 0 : \alpha_{\phi}(W_F^{r+})\; \textrm{has trivial image in } M^\vee\}.
\]
\end{defn}
Note that this is well-defined, i.e. independent of the choice of the representative $\alpha_{\phi}$ of $[\alpha_{\phi}]$. Write $\alpha = \alpha_{\phi}$ and choose another cocycle $\beta$ 
representing $[\alpha_{\phi}]$.  
Then, since $\alpha$ and $\beta$ are cohomologous in $Z^1(W_F, M^\vee)$, there exists $m \in M^\vee$  such that
\[
\beta(w) = m^{-1} \cdot \alpha(w) \cdot {}^w m
\]
for all $w \in W_F$.   By the triviality of the action of $W_F^{0+}$ on $M^\vee$, we have 
\[
\beta(w) = m^{-1} \cdot \alpha(w) \cdot m
\]
for all $w \in W_F^{0+}$.   Therefore, for $r \geq 0$, $\alpha(W_F^{r+})$ has trivial image in $M^\vee$ if and only if so does $\beta(W_F^{r+})$.   

We need to reconcile these two definitions of depth.   

\begin{lem} \label{Oi}  For tamely ramified groups,  these two definitions are equivalent: 
\end{lem}

\begin{proof} It suffices to check that the following are equivalent for $r \geq 0$:
\begin{itemize}
\item $\alpha_{\phi}(W_F^{r+})$ \; has trivial image in \; $M^\vee$,
\medskip
\item $[\alpha_{\phi}]$ belongs to $\H^1(W_F/W_F^{r+}, (M^\vee)^{W_F^{r+}})$.
\end{itemize}
Since $M(F)$ is tamely ramified, the wild inertia group $\cP_F$ acts trivially on $M^\vee$ and so the fixed set $(M^\vee)^{W_F^{r+}}$ equals $M^\vee$ for any $r \geq 0$.  
Therefore $\H^1(W_F/W_F^{r+}, (M^\vee)^{W_F^{r+}})$ is nothing but $\H^1(W_F/W_F^{r+}, M^\vee)$ which is the subset of $\H^1(W_F, M^\vee)$ consisting of cohomology classes which 
can be represented by a $1$-cocycle whose restriction to $W_F^{r+}$ is trivial.   Thus the above two conditions are equivalent.    
\end{proof}

We note that definition \ref{Oidef} is well-adapted to the proofs in \cite{MiPa}.

Lemma \ref{Oi} shows that we now have a consistent definition of depth.

\begin{rem} \label{rem:depthGL}
{\rm In the special case when $M$ is $F$-split, the group $W_F$ acts trivially on $M^\vee$, and $\alpha_\phi$ is a homomorphism, which, by definition, coincides with the restriction of $\phi$ to $W_F$.  Hence Lemma~\ref{Oi} shows that $\dep(\phi)$ coincides with the definition of the depth of $\phi$, as defined for instance in \cite[\S2.3]{ABPS}.}
\end{rem}

Our next result is a refinement of the isomorphism  (\ref{Sh2}).   

\begin{thm}\label{shr}      If $r \geq 0$, then we have a canonical isomorphism  
\begin{eqnarray*}
 \H^1\left(W_F/ W_F^{r+}, (M^\vee)^{W_F^{r+}} \right)  & \cong & \H^1\left(W_E/W_E^{\psi(r)+}, (G^\vee)^{W_E^{\psi(r)+}}\right)
 \end{eqnarray*} 
 where $\psi = \psi_{E/F}$.    
\end{thm}
   
 \begin{proof}   We have the following isomorphisms of complex reductive groups:
 \begin{eqnarray*}
  (M^\vee)^{W_F^{r+}}   &\cong & (\Ind_{W_E}^{W_F} G^\vee)^{W_F^{r+}} \\
  & \cong &\Ind_{W_E/W_E \cap W_F^{r+}}^{W_F/W_F^{r+}} (G^\vee)^{W_F^{r+} \cap W_E}  \\  
  &\cong &\Ind_{W_E/W_E^{\psi(r)+}}^{W_F/W_F^{r+}} (G^\vee)^{W_E^{\psi(r)+}}.
   \end{eqnarray*}
 In this proof, we have used, successively
 \begin{itemize}
\item  the construction of $M^\vee$ as an induced group,
\item the submodule lemma \ref{ABC} with $H = W_E$, $J = W_F$, $M = G^\vee$, $A = W_F^{r+}$,
\item the comparison lemma \ref{comp}. 
 \end{itemize} 
 
Now apply the canonical  isomorphism (\ref{Sh3}). 
 \end{proof}

\begin{thm}\label{depsh}  We have $\dep(\Sh(\phi)) = \psi_{E/F} (\dep(\phi))$.   In particular, $\phi$ has depth $0$ if and only if $\Sh(\phi)$ has depth $0$.   
\end{thm}

\begin{proof} This follows immediately from Theorem \ref{shr} and definition \ref{Oidef}.   
\end{proof}

% We will write $W^+ = W^{0+}, \; I_F^+ = I_F^{0+}$.  
\section{Depth-comparison under the local Langlands correspondence for Weil-restricted groups} \label{sec:LLC}
We assume that the $K$-group $H$ is quasi-split. Let  $\rZ(H)$ and $\rZ(H^\vee)$ denote  the center of $H$ and $H^\vee$, respectively.

\begin{defn} \label{defn:LLC}
A local Langlands correspondence (or $\LLC$) for $(H,K)$ is a surjective map
\[
\lambda_H \colon \Pi(H,K) \to \Phi(H,K),
\]
which satisfies 
%the conditions (desiderata) laid down in \cite[\S 10]{Bor}. 
%With $\phi \in \Phi(H,K)$, the pre-image of $\phi$ via $\lambda_H$ is an $L$-packet denoted $\Pi_{\phi}$.  
the conditions laid down by Langlands in \cite[\S 3]{Lan}.  
\end{defn}

These conditions are the desiderata of Borel \cite{Bor}.
We  will recall them now. 
The parameter $\phi$ determines a character $\chi_{\phi}$ as in \cite{Lan}.  Given $\pi \in \Pi(H,K)$, an element $\alpha \in \H^1(W_K, \rZ(H^\vee))$ determines
an element $\pi_\alpha \in \Pi(H,K)$, see \cite[p.20]{Lan}.     

To every $\phi$  in $\Phi(H,K)$, the pre-image of $\phi$ via $\lambda_H$ is 
%we are going to associate 
a finite but nonempty set $\Pi_\phi$ in $\Pi(H,K)$ such that the following conditions are satisfied.

(i)  If $\phi \neq \phi'$ then $\Pi_\phi$  and $\Pi_{\phi'}$ are disjoint.

(ii) If $\pi \in \Pi_\phi$  then 
\[
\pi(z) = \chi_\phi(z)I, \quad \quad z\in \rZ_H(K).
\]

(iii) If $\phi'  = \alpha \phi$ with  $\alpha \in \H^1(W_K, \rZ(H^\vee)$, then
\[
\Pi_{\phi'}  =  \{\pi_\alpha \otimes \pi | \pi \in \Pi_\phi\}.
\]

(iv) If $\eta \colon H' \to H$ has abelian kernel and cokernel, if $\phi \in \Phi(H,K)$ and $\phi' = \eta^*(\phi)$ then the pullback of any $\pi \in \Pi_\phi$ to $G'(K)$ is the direct sum of finitely many irreducible, quasi-simple representations, all of which lie in $\Pi_{\phi'}$.

(v) If $\phi \in \Phi(H,K)$ and  one element of $\Pi_\phi$  is square integrable modulo $(\rZ(H))(K)$  then all elements are. 
%This happens if and only if ?(WF ) is contained in no proper parabolic subgroup of ?
%if ? = ? ? ?? where ? is one-dimensional and where ??, which operates on V ?, is such that |f?(??(g)v?)|2 is an integrable function on ZG(F)G(F) for any K-finite v? ? V? and any K-finite linear form f? on V ?.

(vi) If $\phi \in \Phi(H,K)$ and one element of $\Pi_\phi$  is tempered then all elements are. With respect to a distinguished splitting, write 
$\phi(w) = (a(w),w)$.   The elements of $\Pi_\phi$  are tempered if and only if $\{a(w) : w \in W_K\}$  is relatively compact in $H^\vee$.   

\begin{rem} {\rm
Note that, although Langlands in \cite{Lan} is working primarily with the extension  $\C/ \R$, he explicitly writes that many of his results hold more generally for finite extensions $E/F$ of local fields,
see \cite[p.7]{Lan}.   In this generality, he proves that the map $\phi \mapsto \chi_{\phi}$ respects restriction of scalars, see \cite[Lemma 2.11]{Lan}.   
The map $\alpha \mapsto \pi_\alpha$ also respects restriction of scalars, see \cite[Lemma 2.12]{Lan}. }
\end{rem}

\begin{defn} \label{lambda}    Let $E/F$ be a finite and Galois extension.     Let $G$ be a connected reductive group defined over $E$ which admits a $\LLC$, say $\lambda_G$. 
 Then the map $\lambda_M$ is defined to be the unique  map  for which the following diagram commutes
 \begin{eqnarray} \label{Diagram}
\begin{CD} \Pi(M,F)@ > \lambda_M >> \Phi(M,F)\\
@V \iota^* VV                       @VV \Sh V\\
\Pi(G,E)@> \lambda_G >> \Phi(G,E)
\end{CD}
\end{eqnarray}
where  $\Sh$ is the restriction to $\Phi(M,F)$ under the injection~\rm{(\ref{eqn:PhiH1})} of the Shapiro isomorphism, as in \cite[8.4]{Bor}.
\end{defn}

The pre-image via $\lambda_M$ of $\Sh^{-1}(\phi)$ will be denoted $\Pi_{\Sh^{-1}\phi}$.   Since the two vertical maps are bijective, it is clear 
that we have equality of cardinalities:
\[
\card(\Pi_{\Sh^{-1}\phi}) = \card(\Pi_\phi).
\]

The map $\lambda_M$ satisfies all the above conditions, and hence it qualifies as a local Langlands correspondence.

\smallskip

A LLC for $(H,K)$ can be enhanced in the following way. 
Let $H^\vee_\sc$ be the simply connected covering of the derived group of $H^\vee$, and  $\rZ(H^\vee_\sc)$ be the center of $H^\vee_\sc$. Let $H^\vee_\ad$ be the adjoint group of $H^\vee$.
Let $\phi\in\Phi(H,K)$. We denote by $\rZ_{H^\vee}(\phi)$ denote the centralizer of $\phi(W_K')$ in $H^\vee$.
Since $\rZ_{H^\vee}(\phi) \cap \rZ(H^\vee) = \rZ(H^\vee)^{W_K}$, we get
$\rZ_{H^\vee}(\phi) / \rZ(H^\vee)^{W_K} \cong \rZ_{H^\vee}(\phi) \rZ(H^\vee) / \rZ(H^\vee)$.
The latter can be considered as a subgroup of the adjoint group $H^\vee_\ad$. Let $\rZ^1_{H^\vee_\sc}(\phi)$ be its inverse image under the
quotient map $H^\vee_\sc \to H^\vee_\ad$. 

Following Arthur \cite[(3.2)]{ArNote}, we consider the component group of $\rZ_{H^\vee}^1(\phi)$:
\begin{equation} \label{eqn:Sphi}
\sS_\phi := \pi_0 \big( \rZ^1_{H^\vee_{\sc}}(\phi) \big) .
\end{equation}
An \textit{enhancement} of $\phi$ is an irreducible representation $\rho$ of $\sS_\phi$.
Via the canonical map $\rZ(H^\vee_\sc) \to\rZ(\sS_\phi)$, every enhancement $\rho$ determines a character $\zeta_\rho$ of $\rZ(H^\vee_\sc)$.

%and by $\cA_\phi$ the full pre-image of $A_\phi/\rZ(G^\vee)^{W_E}\subset G^\vee_\ad$ in $G^\vee_\sc$. 
On the other hand, the group $H$ is an inner twist of a unique quasi-split $K$-group $H^*$. The parametrization of equivalence classes of inner twists of $H^*$ by
\begin{equation} \label{eqn:Kot}
\H^1_c(W_K,H_\ad)\simeq \Irr\left(\rZ(H_\ad^\vee)^{W_K}\right)\end{equation}
provides a character $\zeta_H$ of $\rZ(H_\ad^\vee)^{W_K}$.
We choose an extension to a character $\zeta_H$ of $\rZ(H_\sc^\vee)$. (Such an extension is determined by an explicit construction of $H$ is inner twist of $H^*$.) Then we say that 
$(\phi,\rho)$ (or $\rho$)  is  $H(K)$-relevant if  $\zeta_\rho=\zeta_H^+$.

\begin{defn} \label{defn:enh}
A pair $(\phi,\rho)$, where $\phi$ is a Langlands parameter for $H(K)$ and $\rho$ is an $H(K)$-relevant irreducible representation of the group $\sS_\phi$ defined in {\rm (\ref{eqn:Sphi})}, is called an \textit{enhanced  $L$-parameter} for $(H,K)$.  We denote by $\Phi^\enh(H,K)$ the set of  $H^\vee$-orbits of enhanced $L$-parameters for $(H,K)$ for the following action of $H^\vee$:
\[h\cdot(\phi,\rho)=(h\phi h^{-1},\rho\circ\Ad(h^{-1}))\quad\text{for $h\in H^\vee$.}\]
\end{defn}

\begin{rem} {\rm
A notion of \textit{cuspidality} for enhanced  $L$-parameters was defined in \cite[Definition~6.9]{AMS}. Cuspidal $H(K)$-relevant enhanced  $L$-parameters are expected to parametrize the supercuspidal smooth irreducible representations of $H(K)$ (see \cite[Conjecture~6.10]{AMS}).}
\end{rem}

It is natural to request that  a LLC
\[
\begin{matrix}\lambda\colon& \Pi(H,K)&\to&\Phi(H,K) \cr
&\pi&\mapsto& \phi_\pi
\end{matrix} ,
\]  
may be \textit{enhanced} so that we obtain a bijection
\begin{equation} 
\label{eqn:lambdae} 
\begin{matrix}\lambda^\enh\colon& \Pi(H,K)&\to&\Phi^\enh(H,K) \cr
&\pi&\mapsto& (\phi_\pi,\rho_\pi)
\end{matrix} .
\end{equation}
For any $\phi\in\Phi(H,K)$, the elements in the $L$-packet $\Pi_{\phi}$ will then be parametrized by the set of isomorphism classes of $H(K)$-relevant  irreducible representations of the finite group $\sS_\phi$.

Feng, Opdam and Solleveld proved that the map $\Sh\colon \Phi(M,F)\to\Phi(G,E)$ extends naturally to a bijection 
\begin{equation} \label{eqn:She}
\Sh^\enh\colon\Phi^\enh(M,F)\to \Phi^\enh(G,E),
\end{equation}
and that $\Sh^\enh$ preserves cuspidality (see \cite[Lemma A4]{FOS}). 

\begin{defn} \label{defn:lambdae}  If there exists a bijection $\lambda_G^\enh\colon\Pi(G,E)\to\Phi^\enh(G,E)$,
 then the map $\lambda_M^\enh$ is defined to be the unique  map  for which the following diagram commutes
 \begin{eqnarray} \label{Diagram}
\begin{CD} \Pi(M,F)@ > \lambda_M^\enh >> \Phi^\enh(M,F)\\
@V \iota^* VV                       @VV \Sh^\enh V\\
\Pi(G,E)@> \lambda_G^\enh >> \Phi^\enh(G,E).
\end{CD}
\end{eqnarray}
\end{defn}
By construction $\lambda^\enh_M$ is a bijection and enhances the map $\lambda_M$ defined in Definition~\ref{lambda}.

\smallskip

We assume that $G(E)$ is quasi-split, that is, there is a Borel subgroup of $G$ defined over $E$.
Recall that a Whittaker datum for $G(E)$ is a $G(E)$-conjugacy class of pairs $(B,\theta)$, where $B$ is a Borel subgroup of $G$ defined over 
$E$ with unipotent radical $U$, and $\theta$ is a non-degenerate character $U(E) \to \C^\times$.    Whittaker datum $\fw = (B,\theta)$, an admissible representation
$\pi \in \Pi(G(E))$ is called $\mathfrak{w}$-generic if $\Hom_{U(E)}(\pi, \theta) \neq 0$.  

 We attempt to lift $\fw$ from $G(E)$ to $M(F)$.   Note first that the Weil-restricted group $(\mathfrak{R}_{E/F} B)$
 is a Borel subgroup of $M$, see \cite[\S 5.2]{Bor}.  Thus $M$ is a quasi-split $F$-group. We know that  $B(E)$ and $(\fR_{E/F} B)(F)$ are isomorphic topological groups. 

 We have an exact sequence $1\to U\to B\to T\to 1$, with $T$ a maximal torus. It gives an exact sequence 
 \[1\to \fR_{E/F}U\to \fR_{E/F}B\to \fR_{E/F}T\to 1,\]
 as checked for instance in \cite[A.3.2]{Oes}. Since $\fR_{E/F} T$ is maximal torus of $\fR_{E/F}B$, we have an exact sequence
 \[1\to \Ru(\fR_{E/F}B)\to \fR_{E/F}B\to \fR_{E/F}T\to 1,\]
 where $\Ru(\fR_{E/F}B)$ denotes the unipotent radical  of $\fR_{E/F} B$
 It follows that \[\fR_{E/F}U=\Ru(\fR_{E/F}B). \]
 So $\theta$ is a well-defined non-degenerate character of $(\fR_{E/F} U)(F)=U(E)$. This leads to the following definition:
\begin{equation} \label{eqn:Witt}
\fw_{E/F}: = (\fR_{E/F}B,\theta).
\end{equation}
Then $\fw_{E/F}$ is a Whittaker datum for the Weil-restricted group $M(F)$.  

Since $G(E)$ and $M(F)$ are isomorphic as topological groups, we have 
 \[\Hom_{(\fR_{E/F} U)(F)}(\pi, \theta) = \Hom_{U(E)}(\iota^* \pi, \theta).\]

Therefore, $\pi$ is $\fw_{E/F}$-generic if and only if $\iota^* \pi$ is $\fw$-generic.  
In particular, the set $\Pi_{\phi}$ contains a unique $\mathfrak{w}$-generic constituent if and only if 
the set $\Pi_{\Sh^{-1}(\phi)}$ contains a unique $\mathfrak{w}_{E/F}$-generic constituent, in conformity with 
Conjecture C in Kaletha's article \cite{Kal}.

\smallskip

We write  $\sR_\phi := \pi_0 \big( \rZ_{H^\vee}(\phi) / \rZ(H^\vee)^{W_K} \big)$. The map $H^\vee_\sc \to H^\vee_\ad$ induces a homomorphism $\sS_\phi \to \sR_\phi$ and   $\sS_\phi$ is a central extension of $\sR_\phi$ by  $\rZ(H^\vee_\sc) / \rZ(H^\vee_\sc) \cap \rZ_{H^\vee_\sc}(\phi)^\circ$ (see \cite[Lemma 1.7]{ABPSConj}).

From now on we assume that the characteristic of $E$ is zero, that $G(E)$ is quasi-split, and that a Whittaker datum $\fw=(B,\theta)$ for $G(E)$ is fixed.  
Then the expected parametrization  reduces to bijections
\begin{equation} \label{eqnLL_G}
\begin{matrix}
\bij_\phi\colon& \Pi_\phi&\to &\Irr(\sR_\phi)\cr
&\pi&\mapsto&\rho_\pi
\end{matrix},
\end{equation}
for all $\phi\in \Phi(H,K)$, where $\Irr(\sR_\phi)$ denotes the set  irreducible characters of $\sR_\phi$.

Then we can form for any $\phi\in\Phi_\bd(H,K)$ and $r\in \sR_\phi$ the virtual character
\begin{equation} \label{eqn:virtual}
\Theta_\phi^r:=\sum_{\pi\in \Pi_\phi}\,(\bij_\phi(\pi))(r)\,\Theta_{\pi},
\end{equation}
where $\Theta_{\pi}$ is the Harish-Chandra distribution character of $\pi$. 

As observed in \cite[(A.22)]{FOS},  for any $\phi\in\Phi(M,F)$, we have a canonical isomorphism
\begin{equation} \label{eqn:RR}
{}^L \iota\colon\sR_{\Sh(\phi)} \isom \sR_\phi.
\end{equation}
We define a bijection
\begin{equation} \label{eqn:Liota}
{}^L \iota^*\colon \Irr(\sR_\phi)\isom\Irr(\sR_{\Sh(\phi)}),
\end{equation}
by setting
\begin{equation} \label{eqn:equl} 
({}^L\iota^*(\rho))(r'):=\rho({}^L \iota(r')),\quad \text{for any $\rho\in \Irr(\sR_\phi)$ and any $r'\in\sR_{\Sh(\phi)}$.}
\end{equation}

\begin{prop} \label{prop:HC}
Let $\phi\in \Phi(M,F)$. We assume that there exists a bijection $\bij_\phi$ as in (\ref{eqnLL_G}). 
Then, for any $f\in\sC_c^\infty(M(F))$, we have 
\[\Theta^r_\phi(f)=\Theta^{{}^L\iota(r)}_{\Sh(\phi)}(\iota^*f), \quad \text{for any $r\in\sR_\phi$,}\]
where $\iota^*f\colon G(E)\to \C$ is the function defined by 
$(\iota^* f)(g):=f(\iota(g))$ for $g\in G(E)$.
\end{prop}
\begin{proof}
Let $\bij_{\Sh(\phi)}\colon\Pi_{\Sh(\phi)}\to \Irr(\sR_{\Sh(\phi)})$ denote the unique map which 
makes the following diagram commutative:
\begin{eqnarray} \label{DiagramRR}
\begin{CD} \Pi_\phi@ > \bij_\phi>> \Irr(\sR_\phi)\\
@V \iota^* VV                       @VV {}^L \iota^* V\\
\Pi_{\Sh(\phi)} @ > \bij_{\Sh(\phi)} >> \Irr(\sR_{\Sh(\phi)})
\end{CD} .
\end{eqnarray}
Let $r\in\sR_\phi$. We write $r':={}^L\iota^{-1}(r)$. Then we obtain that
\[\Theta_{\Sh(\phi)}^{r'}=\sum_{\pi'\in \Pi_{\Sh(\phi)}}(\bij_{\Sh(\phi)}(\pi'))(r')\,\Theta_{\pi'}=
\sum_{\pi\in \Pi_{\phi}}(\bij_{\Sh(\phi)}(\iota^*\pi))(r')\,\Theta_{\iota*\pi}.\]
We observe that, for any $\pi\in \Pi(M,F)$, we have, for any $f\in\sC_c^\infty(M(F))$:
\begin{equation} \label{eqn:HCchar}
\Theta_{\pi}(f)=\Theta_{\iota^*\pi}(\iota^*f).
\end{equation}
Using the commutativity of the diagram (\ref{DiagramRR}), we get
\[\Theta_{\Sh(\phi)}^{r'}(\iota^* f)=\sum_{\pi\in \Pi_{\phi}}({}^L\iota^*(\bij_{\phi}(\pi)))(r')\,\Theta_{\iota^*\pi}(\iota^*f).\]
By using (\ref{eqn:equl}) and (\ref{eqn:HCchar}), we finally get that
$\Theta_{\Sh(\phi)}^{{}^L\iota(r)}(\iota^* f)=\Theta_\phi^r(f)$, for any $r\in\sR_\phi$.
\end{proof}

\smallskip

\textsc{The Hiraga Ichino Ikeda conjecture \cite{HII}.}
We fix an additive character $\psi\colon K\to \C^\times$ which is trivial on the ring of integers $\fo_K$ and endow $K$ with the Haar measure that gives its ring of integers volume $1$. 

\begin{defn} \label{defn:HII}
A  $\LLC$ correspondence $\lambda_H$ satisfies the $\HII$ conjecture for $\psi$ if for any square-integrable modulo centre representation $\omega$ of $L(K)$,  the formal degree  of $\omega$ is
\[
\fdeg
(\omega)=\dim(\rho_\omega)\,|\sR_{\phi_\omega}|^{-1}\,\gamma(0,\Ad_{H^\vee,L^\vee} \circ\phi_\omega,\psi),
\]
where $\lambda_H^\enh(\omega)=(\phi_\omega,\rho_\omega)\in\Phi^\enh(L,K)$,  and where $\Ad_{H^\vee,L^\vee}$ is the adjoint representation of the group ${}^L L$ on the quotient of the Lie algebra of $H^\vee$ by that of $\rZ(L^\vee)^{W_K}$ and $\gamma(0,\Ad_{H^\vee,L^\vee} \circ\phi_\omega,\psi)$ is the corresponding adjoint $\gamma$-factor.
\end{defn}

\smallskip

\textsc{Transfer.}
We assume from now that $H$ is a quasi-split $K$-group.
A semisimple element in $H(K)$ is called \textit{strongly regular} if its centralizer is a torus.  
We denote by $H(K)_\sr$ the open subvariety of $H(K)$ fsting of the strongly regular semisimple elements. 

Let $\gamma\in H(K)_\sr$ and let $H(K)_\gamma$ denote its centralizer  in $H(K)$.  
Let $f\in\sC_c^\infty(H(K))$. 

The orbital integral $\rO_\gamma(f)$ of $(f,\gamma)$ is 
\[\rO_\gamma(f):=\int_{H_\gamma(K)\backslash H(K)}f(x^{-1}\gamma x)d\dot x,\]
where $d\dot x$ is an invariant measure on the quotient $H(K)_\gamma\backslash H(K)$.

The stable orbital integral $\SO_\gamma(f)$ of $(f,\gamma)$ is
\[\SO_\gamma(f):=\sum_{\gamma'\in S(\gamma)} \rO_{\gamma'}(f),\]
where $S(\gamma)$ is a set of representatives for the $H(K)$-conjugacy classes of $\gamma$ in its $H(K^\sep)$-conjugacy class (so-called the \textit{stable conjugacy} class of $\gamma$).

We recall from \cite[Def.~2]{Kal} that an \textit{extended endoscopic triple} for $(H,K)$ is a triple $\fe=(H_\fe,s,{}^L\eta)$, where:
\begin{itemize}
\item[$\bullet$]
$H_\fe$ is a quasi-split connected reductive $K$-group, 
\item[$\bullet$]
$s$ is a semisimple element in $H^\vee$, 
\item[$\bullet$]
${}^L\eta\colon {}^LH_\fe\to{}^LH$ is an $L$-homomorphism of $L$-groups (as in \cite[\S 15.1]{Bor}) that restricts to an isomorphism of complex reductive groups $H_\fe^\vee\isom \rZ_{H^\vee}(s)^\circ$,  such that ${}^L\eta(h)$ commutes with $s$, for any $h\in {}^LH_{\fe}$. 
\end{itemize}

Let $\fe=(H_\fe,s,{}^L\eta)$ be an extended endoscopic triple for $(H,K)$. We have $\phi(W_K')\subset {}^L\eta(W'_K)$, since ${}^L\eta\colon{}^LH_\fe\to {}^LH$. It follows that 
$s\in\rZ_{H^\vee}(\phi)$.  We denote by $\overline s$ the image of $s$ in $\sR_\phi$. 

Let $\fw$ be a Whittaker datum for $H(K)$.  
We recall that $f_\fe\in\sC_c^\infty(H_\fe(K))$ is called a \textit{transfer} of $f\in\sC_c^\infty(H(K))$ if for all $\gamma\in H_\fe(K)_\sr$ we have
\[\SO_\gamma(f_\fe)=\sum_{\delta}\Delta[\fw,\fe](\gamma,\delta)\,\rO_\delta(f),\] 
where $\delta$ runs over the set of conjugacy classes in $H(K)_\sr$, and  where \[\Delta(\fw,\fe)\colon H_\fe(K)_\sr\times H(K)_\sr\to \C\]
is the Langlands-Shelstad \textit{transfer factor} associated to $\fw$.

\begin{defn} \label{defn:refined LLC}
An enhanced $\LLC$ for $(H,K)$ is a bijection
\[\begin{matrix}\lambda_H^\enh \colon & \Pi(H,K) &\to& \Phi^\enh(H,K)\cr
&\pi&\mapsto& (\phi_\pi,\rho_\pi)
\end{matrix}
\]
such that the map $\lambda_H\colon \pi\mapsto\phi_\pi$ is a $\LLC$, and the following extra properties hold:
\begin{enumerate}
\item[{\rm (1)}] $\lambda_H$ satisfies the $\HII$ conjecture for square integrable modulo center representations.
\item[{\rm (2)}] For any Whittaker datum $\fw$ for $(H,K)$, and all $\phi\in \Phi_\bd(H,K)$ the $L$-packet $\Pi_{\phi}:=\lambda_H^{-1}(\phi)$ contains a unique $\mathfrak{w}$-generic constituent.
\item[{\rm (3)}] $\lambda_H^\enh$ restricts  to a bijection from the set of isomorphism classes of supercuspidal irreducible representations of $H(K)$ to the set of $H^\vee$-conjugacy classes of cuspidal enhanced $L$-parameters.
\item[{\rm (4)}] For any $\phi\in \Phi(H,K)$, the map 
\[\begin{matrix}\bij_\phi\colon &\Pi_\phi&\to& \Irr(\sS_\phi)\cr
&\pi&\mapsto& \rho_\pi\end{matrix}\]
is a bijection,
\item[{\rm (5)}] 
When $H$ is quasi-split over $K$, for any extended endoscopic triple $\fe=(H_\fe,s,{}^L\eta)$ for $(H,K)$, there exists a bijection
\[\lambda_{H_\fe}^\enh \colon \Pi(H_\fe,K) \to \Phi^\enh(H_\fe,K)\] 
which satisfies the analogs of {\rm (1)}--{\rm (4)} for $(H_\fe,K)$, and a Whittaker datum $\fw_\fe$ for $(H_\fe,K)$, such  that
\begin{itemize} 
\item[(a)]  for $\phi_\fe\in \Phi_\bd(H_\fe,K)$,  the character $\rho_{\pi^\gen_\fe}$ is trivial if  $\pi_\fe^\gen$ is the $\fw_\fe$-generic constituent of $\Pi_{\phi_\fe}$,
\item[(b)]  for any pair $(f_\fe,f)\in \sC_c^\infty(H_\fe(K))\times \sC_c^\infty(H(K))$ of functions such that $f_\fe$ is a transfer of $f$, we have the equality
 \[
 \Theta_{\phi_\fe}^1(f_\fe)=\Theta_{{}^L\eta\circ\phi_\fe}^{\overline{s}}(f).
 \]
 \end{itemize}
 \end{enumerate}
 \end{defn}

\begin{thm} \label{Main1}   Consider a bijection
\[
\lambda_G^\enh\colon\Pi(G,E) \longrightarrow \Phi^\enh(G,E).
\]
Then

\begin{enumerate}
\item[{\rm (i)}] The map $\lambda_M^\enh$ defined in~{\rm (\ref{defn:lambdae})} is an enhanced $\LLC$ for $(M,F)$ if and only if $\lambda_G^\enh$ is an enhanced $\LLC$ for $(G,E)$.
\item[{\rm (ii)}] Furthermore, if $\pi \in \Pi(M,F)$  and $\dep(\lambda_G(\iota^* \pi))  =  \kappa_\pi \cdot \dep(\iota^* \pi)$ then we have 
\[
 \dep(\lambda_M(\pi))  = \varphi_{E/F}(\kappa_\pi \cdot e \cdot \dep(\pi)),
 \]
where $\varphi_{E/F}$ is the Hasse-Herbrand function. 
\end{enumerate}
\end{thm} 

\begin{proof}   
It is proved in  \cite[Proposition~A.7]{FOS} that, for any finite separable field extension $E/F$,  the HII conjecture holds for $\omega$ a square-integrable modulo center irreducible representation of $M(F)$  if and only if  its holds for $\iota^*(\omega)$: it shows that $\lambda_G$ satisfies 
Definition~\ref{defn:refined LLC}~(1) if and only if $\lambda_M$ satisfies it.
  
We have already seen that Definition~\ref{defn:refined LLC}~(2)  is satisfied by $\lambda_G$ if and only if it is satisfied by $\lambda_M$.  Since $\Sh^\enh$ preserves the cuspidality, Definition~\ref{defn:refined LLC}~(3) is satisfied by $\lambda_G^\enh$ if and only if it is satisfied by $\lambda_M^\enh$.  

We write $\lambda_M^\enh(\pi)=(\phi_\pi,\rho_\pi)$ for $\pi\in \Pi(M,F)$ and $\bij_\phi(\pi)=\rho_\pi$ for $\phi=\phi_\pi$. Then Definition~\ref{defn:refined LLC}~(4) is satisfied by  $\bij_\phi$ if and only if it is satisfied by $\bij_{\Sh(\phi)}$, where
\[\bij_{\Sh(\phi)}(\iota^*\pi):={}^L\iota^*(\rho_{\pi}),\]
and ${}^L \iota^*$ is the natural bijection between $\Irr(\sS_\phi)$ and $\Irr(\sS_{\Sh(\phi)})$.

When $G$ is qusi-split over $E$, let $\fw$ be a Whittaker datum for $(G,E)$, and let $\fw_{E/F}$ be its lift to $(M,F)$ as in (\ref{eqn:Witt}). Let $\phi\in\Phi_\bd(M,F)$. Then the  diagram~(\ref{DiagramRR}) shows  that the existence of a bijection $\bij_\phi\colon \Pi_\phi\to \Irr(\sR_\phi)$ satisfying Definition~\ref{defn:refined LLC}~(3).a  is equivalent to the existence of a bijection $\bij_{\Sh(\phi)}\colon\Pi_{\Sh(\phi)}\to \Irr(\sR_{\Sh(\phi)})$ satisfying Definition~\ref{defn:refined LLC}~(5).a.  Indeed, as already observed, $\pi^\gen\in\Pi_\phi$ is $\fw_{E/F}$-generic if and only if $\iota^* \pi^\gen$ is $\fw$-generic. On the other hand, 
 the diagram~(\ref{DiagramRR}) implies that $\bij_\phi(\pi^\gen)$ is the trivial character of $\sR_\phi$ if an only if $\bij_{\Sh(\phi)}(\iota^*\pi^\gen)$ is the trivial character of $\sR_{\Sh(\phi)}$, since  (\ref{eqn:equl}) shows that ${}^L \iota^*$ maps the trivial character of $\sR_\phi$ to that of $\sR_{\Sh(\phi)}$.

We fix a $\Gamma_E$-stable Borel pair $(B^\vee,T^\vee)$ in $G^\vee$ and an extended endoscopic triple $\fe=(G_\fe,s,{}^L\eta)$ for $(G,E)$ such that $s'\in T^\vee$.
The following construction is based on \cite[\S~1.2]{LW}.  
Every Borel pair in $M^\vee$ can be written as $(I_{\Gamma_E}^{\Gamma_F}(B^\vee), I_{\Gamma_E}^{\Gamma_F}(T^\vee))$ for a well-determined Borel pair $(B^\vee,T^\vee)$ in $G^\vee$. It is $\Gamma_F$-stable if and only if the pair $(B^\vee,T^\vee)$ is $\Gamma_E$-stable.
Setting $\cW^\vee:=\Nor_{G^\vee}(T^\vee)/T^\vee$, we have 
\[\Nor_{M^\vee}(\fR_{E/F}(T^\vee))/\fR_{E/F}(T^\vee)=I_{\Gamma_E}^{\Gamma_F}(\cW^\vee).\]
We denote by $\tau\mapsto\tau_G$ the natural action of $\Gamma_E$ on $G^\vee$ and we
 set $B_\fe^\vee:=B^\vee\cap G_\fe^\vee$. Then $(B_\fe^\vee,T^\vee)$ is a Borel pair in $G_\fe$.
Let $\tau\in \Gamma_E$. For each $v'\in W_E$ in the inverse image of $\tau$ under the natural map $W_E\to\Gamma_E$, we choose an element $(g(v'),v')$ of $G^\vee_\fe\rtimes W_E={}^LG_\fe$ such that the automorphism $\Int_{g(v')}\circ\tau_G$ preserves the Borel pair $(B_\fe^\vee,T^\vee)$.
The coset  $T^\vee g (v')$ being well-determined, there is a well-determined element
$w_{G_\fe} (\tau)$ of $\cW^\vee$ such that the conjugacy action of $(g(v'),v')$ on $T^\vee$ is given by  $w_{G_\fe} (\tau)\tau_G$, where we identify 
$w_{G_\fe} (\tau)$ with the automorphism $\Int_{w_{G_\fe} (\tau)}$ of $T^\vee$. 
The map $\tau\mapsto w_{G_\fe} (\tau)$ is a $1$-cocycle. Let $\alpha\in \rH^1(\Gamma_E,\cW^\vee)$ denote its cohomology class, and
let $\Sh^{-1}(\alpha)$ be the inverse image of $\alpha$ under the Shapiro isomorphism
\[\Sh\colon \rH^1(\Gamma_F,I_{\Gamma_E}^{\Gamma_F}(\cW^\vee))\to\rH^1(\Gamma_E,\cW^\vee).\]
We choose a $1$-cocycle $w_{G_\fe,M}$ of $\Gamma_F$ with values in $I_{\Gamma_E}^{\Gamma_F}(\cW^\vee)$ which belongs to the cohomology class of $\alpha$. Up to replacing $w_{G_\fe,M}$ by a cohomologous $1$-cocycle, we may, and do, assume that
\begin{equation} \label{eqn:wG}
(w_{G_\fe,M}(\tau))(1)=w_{G_\fe}(\tau)\,\quad\text{for any $\tau\in\Gamma_E$}.
\end{equation}
Let $\sigma\mapsto \sigma_{G_\fe,M}$ be the action of $\Gamma_F$ on the torus $I_{\Gamma_E}^{\Gamma_F}(T^\vee)$ of $M^\vee$ defined by
\begin{equation} \label{eqn:action}
\sigma_{G_\fe,M}:=w_{G_\fe,M}(\sigma)\,\sigma_M,
\end{equation}
where $\sigma\mapsto\sigma_M$ is the natural action of $\Gamma_F$ on $M^\vee$. 
It allows to define an application $s_{E/F}$ from $\Gamma_F$ to $T^\vee$ by sending $\sigma$ to 
\begin{equation} \label{eqn:semisimples}
s_{E/F}(\sigma):=w_{G_\fe,M}(\sigma)(1)^{-1}\,(s).
\end{equation}
In particular,  $s_{E/F}$ belongs to $I_{\Gamma_E}^{\Gamma_F}(T^\vee)$, 
%, that is, that 
%\begin{equation} \label{eqn:sEF}
%s_{E/F}(\tau\sigma)=\tau_G(s_{E/F}(\sigma))\quad\text{for any $\tau\in\Gamma_E$ and $\sigma\in \Gamma_F$,}
%\end{equation}
is fixed by the action $\sigma\mapsto \sigma_{G_\fe,M}$ and we have  $s_{E/F}(1)=s$. We set 
\begin{equation} \label{eqn:Me}
M_\fe^\vee:=\rZ_{M^\vee}(s_{E/F})^\circ.
\end{equation}
Then $M_\fe^\vee\cap I_{\Gamma_E}^{\Gamma_F}(B^\vee)$ is a Borel subgroup of $M^\vee_\fe$. For each $v\in W_F$ with image $\sigma$ in $\Gamma_F$, we choose a representative $\tilde w_{G_\fe,M}(v)=\tilde w_{G_\fe,M}(\sigma)$ of $w_{G_\fe,M}(\sigma)$ in $\Nor_{M^\vee}(I_{\Gamma_E}^{\Gamma_F}(T^\vee))$. 

The automorphism $\Int_{\tilde w_{G_\fe,M}(\sigma)}\circ\tau_M$ preserves the Borel pair $(M_\fe^\vee\cap I_{\Gamma_E}^{\Gamma_F}(B^\vee),I_{\Gamma_E}^{\Gamma_F}(T^\vee))$ of $M^\vee_\fe$. 
We define
\begin{equation} \label{eqn:calM}
\sM_\fe:=\left\{(m\,\tilde w_{G_\fe,M}(\sigma)(v),v)\,:\,m\in M_\fe^\vee, v\in W_F\right\}\,\subset\, {}^LM.
\end{equation}
The set $\sM_\fe$ is a group which normalizes $M_\fe^\vee$. Thus we can deduce an $L$-action of $\Gamma_F$ on  $M^\vee_\fe$, and hence form the semidirect product $M^\vee_\fe\rtimes W_F$. Then let $M_\fe$ be a quasi-split connected reductive  $F$-group which has $M^\vee_\fe\rtimes W_F$ as $L$-group. We have $M_\fe=\fR_{E/F}(G_\fe)$.
Let $\lambda_{G_\fe}$ be a LLC for $(G_\fe,E)$. Then the arguments above, applied to $(G_\fe,E)$, show that $\lambda_{G_\fe}$ satisfies the properties (1), (2) and (3).a. 

Let ${}^L\eta_{E/F} \colon {}^LM_\fe\to{}^L M$ be an $L$-homomorphism of $L$-groups such that $\fe_{E/F}=(M_\fe,s_{E/F},{}^L\eta_{E/F})$ is an extended endoscopic triple for $(M,F)$, then 
the diagram
\begin{equation} \label{diag}
\begin{CD} \Phi(M_\fe,F)@ > \Phi({}^L\eta_{E/F})>> \Phi(M,F)\\
@V \Sh VV                       @VV  \Sh V\\
\Phi(G_\fe,E) @ > \Phi({}^L\eta) >> \Phi(G,E)
\end{CD} 
\end{equation}
 where 
\[
\Phi({}^L\eta)\colon\phi_\fe\mapsto {}^L\eta\circ\phi_\fe\quad\text{and}\quad\Phi({}^L\eta_{E/F})\colon\phi_\fe'\mapsto {}^L\eta_{E/F}\circ\phi_\fe',
\]
is commutative, that is,  we have ${}^L\eta_{E/F}\circ \Sh(\phi_\fe)=\Sh({}^L\eta\circ\phi_\fe)$ for any $\phi_\fe\in \Phi(M_\fe,F)$.
%  such that 
%$\beta(\phi_\fe)={}^L\eta_{E/F}\circ \phi_\fe$ for any $\phi_\fe\in\Phi(M_\fe,F)$.}
%Attempt to construct  ${}^L\eta_{E/F}$: By definition (see \cite[\S15.1]{Bor}), since ${}^L\eta\colon {}^L G_\fe \to {}^L G$ is an $L$-homomorphism  of $L$-groups, its restriction to $G_\fe^\vee$ is a morphism $\eta^\vee\colon G_\fe^\vee\hookrightarrow G^\vee$. Since $M_\fe^\vee=I_{\Gamma_E}^{\Gamma_F}(G_\fe^\vee)$, its  elements are the functions $m_\fe\colon \Gamma_F\to G_\fe^\vee$ satisfying $m_\fe(\tau\sigma)=\tau\cdot m_\fe(\sigma)$ for any $\sigma\in \Gamma_F$ and any
%$\tau \in \Gamma_E$.
%We define an injective homomorphism $\eta_{E/F}^\vee\colon M^\vee_\fe\hookrightarrow M^\vee$ by setting
%\begin{equation} \label{eqn:etaEF}
%\eta_{E/F}^\vee(m_e):=\eta^\vee\circ m_e\quad\text{for any $m_\fe\in M_\fe^\vee$.}
%\end{equation}}
%Then .

Let $\iota_\fe \colon G_\fe(E)\isom M_\fe(F)$. 
The maps  $\iota$ and $\iota_\fe$ induce bijections $G(E)_\sr\isom M(F)_\sr$ and $G_\fe(E)_\sr\isom M_\fe(F)_\sr$, respectively.  Let $\delta\in G(E)_\sr$ and $\gamma\in G_\fe(E)_\sr$. 
%Recall that $G(E)_\delta$  is the centralizer of $\delta$ in $G(E)$, and $M(F)_{\iota(\delta)}$ the centralizer of $\iota(\delta)$ in $M(F)$. 
We have
%\begin{equation} \label{eqn:centr}
\[\iota(G(E)_\delta)\simeq M(F)_{\iota(\delta)}\quad\text{and}\quad \iota_\fe(G_\fe(E)_\gamma)\simeq M_\fe(F)_{\iota_\fe(\gamma)}.
\]
%\end{equation}
Let $(f,f_\fe)\in \sC_c^\infty(M(F))\times\sC_c^\infty(M_\fe(F))$ such that $f_\fe$ is a transfer of $f$. 
We denote by $\iota^*f\colon G(E)\to \C$ and $\iota_\fe^*f_\fe\colon G_\fe(E)\to \C$ the functions defined by 
\[(\iota^* f)(g):=f(\iota(g))\quad \text{and}\quad(\iota^*_\fe f_\fe)(g):=f_\fe(\iota_\fe(g_\fe)) \quad\text{for $g\in G(E)$ and $g_\fe\in G_\fe(E)$.}\]
We have $\iota^*f\in \sC_c^\infty(G(E))$ and $\iota_\fe^*f_\fe\in \sC_c^\infty(G_\fe(E))$. The transfer factors 
\[\Delta(\fw,\fe)\colon G_\fe(E)_\sr\times G(E)_\sr\to \C\quad\text{and}\quad\Delta(\fw_{E/F},\fe_{E/F})\colon M_\fe(F)_\sr\times M(F)_\sr\to \C\]
coincide (it was observed in \cite[Lemme~5.4]{Wal} in the Lie algebras case).
It follows that $f_\fe$  is a transfer of $f$ if and only if $\iota_\fe^*f_\fe$  is a transfer of $\iota^*f$.

Let $\phi_\fe\in\Phi(M_\fe,F)$. Then the combination of Proposition~\ref{prop:HC}  with the commutative diagram (\ref{diag}) implies that 
\[\Theta^1_{\phi_\fe}(f_\fe)=\Theta_{{}^L\eta_{E/F}\circ\phi_\fe}^{\overline s_{E/F}}(f)\quad \text{if and only if}\quad\Theta^1_{\Sh(\phi_\fe)}(\iota_\fe^*f_\fe)=\Theta_{^L\eta\circ\Sh(\phi_\fe)}^{\overline{s}}(\iota^*f),\]
that is, $\lambda_M$ satisfies Definition~\ref{defn:refined LLC}~(5).b if and only if $\lambda_G$ satisfies it.

The assertion (ii) follows from Corollary \ref{ethm},  Theorem \ref{depsh} and the commutativity of the diagram (\ref{Diagram}). 
\end{proof}

\smallskip

\smallskip

As a special case of Theorem~\ref{Main1}, we have 

\begin{thm}\label{Main2}  If $\lambda_G$ is depth-preserving, then, for all $\pi \in \Pi(M,F)$,  we have
\[
\dep (\lambda_M (\pi)) =  \varphi_{E/F}(e \cdot \dep( \pi))
\]   

In particular, we have
\begin{itemize}
\item $\dep(\lambda_M(\pi)) / \dep(\pi) \to 1 \quad \textrm{as} \quad \dep(\pi) \to \infty$ 
 \smallskip 
\item $\lambda_M$ is depth-preserving if and only if $E/F$ is tamely ramified,
\smallskip
\item if $E/F$ is wildly ramified then, for each  $\pi$ with $\dep(\pi)>0$,
 we have
\[
\dep(\lambda_M(\pi) ) > \dep(\pi).
\]
\end{itemize}
\end{thm}

\begin{proof} This follows from Theorem \ref{Main1} and the properties of $\varphi_{E/F}$ recalled  in \S\ref{sec:Shapiro}.   
\end{proof}

When $G(E) = \GL_1(E)$, Theorem \ref{Main2} strengthens the  main result of \cite{MiPa} for induced tori.   For tamely ramified induced tori, we recover the depth-preservation theorem  of Yu \cite{Yu1}.

 \section{Applications}   \label{sec:App}

\subsection{An inequality between depths} 

\begin{lem} \label{lem:ineq-depth}
Let $M$ and $\tM$ be two reductive $F$-groups such that $\tM$ is $F$-split and there exist an $L$-embedding
$u\colon  {}^LM\to{}^L\tM$ which satisfies the following property:
if $v \in W_F$ acts trivially on $M^\vee$, then we have $u(1,v) = (1,v)$. 

Then, for a given element $(m,v) \in {}^LM $: if  $u(m,v) = (1, v)$, then we have $m = 1$. In particular, for any $\phi\in\Phi(M,F)$, 
we have $u\circ\phi\in\Phi(\tM,F)$ and
\[\dep(u\circ\phi)\le\dep(\phi).\]
\end{lem}
\begin{proof} The conjugation isomorphism $\Int(1, v)$ of ${}^L\tM=M^\vee\rtimes W_F$ is trivial on the first factor $M^\vee$. On the other hand, since we have $u(m, v) = (1, v)$, it should restricts to $\Int(m,v)$ on $M^\vee$. This implies that the isomorphism $\Int(m)\circ [v]$ of $M^\vee$ is the identity, where $[v]$ denotes the action of $v \in W_F$ on $M^\vee$.
Here we recall that the action of $W_F$ on $M^\vee$ is defined by using a fixed pinning. More precisely, we have a canonical exact sequence
\[1\to \Int(M^\vee) \to \Aut(M^\vee) \to \Out(M^\vee) \to 1.\]
As $M$ is defined over $F$, we get an action of $W_F$ on its root data, hence we have a homomorphism $W_F \to \Out(M^\vee )$. ($\Out(M^\vee)$ is nothing but the automorphisms of the root data). As we also have a splitting
$\Out(M^\vee) \to\Aut(M^\vee)$
coming from a fixed pinning, we can get a homomorphism $W_F \to \Out(M^\vee)\to\Aut(M^\vee)$ by sending $v$ to $[v]$.
This was nothing but the definition of the action of $W_F$ on $M^\vee$. Therefore the equality $\Int(m)\circ [v] = \Id_{M^\vee}$ says that $[v] = \Int(m)^{-1} = \Id_{M^\vee}$. Hence, by our assumption, we get $u(1, v) = (1, v)$. As we have $u(m, v) = u(m, 1) \cdot u(1, v)$, the equality $u(m, v) = (1, v) $(and the injectivity of  the restriction of $u$ to $M^\vee$) implies $m = 1$. Then the inequality between the depths of $\phi$ and $u\circ \phi$ follows from Definition~\ref{Oidef}.
\end{proof}

\subsection{Automorphic induction}   Let $n\ge 1$ be an integer and $d=[E:F]$.
Let $G$ be the $E$-group $\GL_n$ and let $\tM$ be the $F$-group $\GL_{nd}$. Both groups $G$ and $\tM$ admits a local Langlands correspondence (see \cite{HT}, \cite{He} or \cite{Sch}), and the corresponding maps $\lambda_G$ and $\lambda_{\tM}$ are  bijective. 

Let $\pi_E\in\Pi(G,E)$. We will denote by $\phi_E\in\Phi(G,E)$ the $L$-parameter of $\pi_E$, that is,  $\phi_E:=\lambda_G(\pi_E)$. It is proved in \cite[\S7.3, Proposition~2]{HeBord} that the $L$-parameter $\tilde\phi$ of the representation $\tilde\pi$ of $\tM$ obtained from $\pi_E$ by \textit{automorphic induction}  (when the latter exists, see \cite{HH})  satisfies
\begin{equation} \label{eqn:AI}
\tilde\phi=\Ind_{W_E'}^{W_F'}(\phi_E).
\end{equation}

\begin{lem} \label{lem:Ind}
Let $\phi_E\in\Phi(\GL_n,E)$. We have 
\[\dep(\Ind_{W_E}^{W_F}(\phi_E))=\varphi_{E/F}(\dep(\phi_E)),\]
where $\varphi_{E/F}$ is the inverse of $\psi_{E/F}$.
\end{lem}
\begin{proof} 
%The map that sends $f\in (\Ind_{W_E}^{W_F}(R))^{W_F^r}$ to $vW_F^r\mapsto f(g)$ defines a canonical isomorphism from 
The restriction to $W_E$ of the $L$-parameter $\phi_E\colon W'_E\to\GL_n(\C)$ is a representation of $W_E$ of space $V$ with $\GL(V)=\GL_n(\C)$. By using Lemma~\ref{ABC} and Lemma~\ref{comp}, we obtain
\[(\Ind_{W_E}^{W_F}(V))^{W_F^{r+}}\simeq\Ind_{W_E/W_E\cap W_F^{r+}}^{W_F/W_F^{r+}}(V^{W_E\cap W_F^{r+}})\simeq \Ind_{W_E/W_E^{\psi_{E/F}(r)+}}^{W_F/W_F^{r+}}(V^{W_E^{\psi_{E/F}(r)+}}).\]
It follows that 
\[(\Ind_{W_E}^{W_F}(V))^{W_F^{r+}}\ne\{0\}\;\iff\; V^{W_E^{\psi_{E/F}(r)+}}\ne\{0\}.\]
Then the result follows from Remark~\ref{rem:depthGL}.
%The Mackey formula gives
%\[\Res_{W_F^r}^{W_F}(\Ind_{W_E}^{W_F}(R))=\bigoplus_{v\in W_F^r\backslash W_F/W_E}(\Ind_{W_F^r\cap W_E}^{W_F^r}\circ\Res^{W_E^r}_{W_F^r\cap W_E})(R^v),\]
%where denotes the $v$-twisted module $R$.
%It follows from Mackey theory and Lemma \ref{inertia}, in a similar way as the proof of \cite[Prop. 2]{MiPa}.
\end{proof}

\begin{thm} \label{thm:depthAI} We have
\[\dep(\tilde\pi)=\varphi_{E/F}(\dep(\pi_E)).\]
\end{thm}
\begin{proof}
Since  $\lambda_{\tM}$ and $\lambda_G$ are depth preserving (see \cite[Theorem 2.9]{ABPS}), using Lemma~\ref{lem:Ind}, we get 
\[\dep(\tilde\pi)=\dep(\tilde\phi)=\dep(\Ind_{W_E'}^{W_F'}(\phi_E))=\varphi_{E/F}(\dep(\phi_E))=\varphi_{E/F}(\dep(\pi_E)).\]
\end{proof}

\subsection{Asai lift}   We take for $G$ the group $\GL_n$ and $M=\fR_{E/F}(G)$. We assume that $[E:F]=2$. 
We denote by $V$ the $2$-dimensional $\C$-vector space $\C^2$.
Hence we have ${}^LG=G^\vee \times W_F$, with $G^\vee=\GL_2(\C)=\GL(V)$, and ${}^L M= M^\vee\rtimes W_F$, where
$M^\vee= \GL_n(\C)\times\GL_n(\C)$, and  $W_F$ permutes the two factors $\GL_2(\C)$ among themselves. 
%We will write \[w\cdot(g_1,g_2)=(g_{w(1)},g_{w(2)}),\quad\textrm{for $g_1,g_2\in \GL(V)$ and $w\in W_F$.}\]
Let $\tG$ denote the group $\GL_{n^2}$.  We have $\tG^\vee=\GL_{n^2}(\C)=\GL(V^{\otimes 2})$ and ${}^L\tG=\tG^\vee \times W_F$. 

Let $r_\rmA \colon {}^L M\to{}^L\tG$ denote the map defined by
\begin{equation} \label{eqn:Asai}
r_\rmA(g_1,g_2,b) := 
\begin{cases}
(g_1\otimes g_2,b)&\text{if $b\in W_E$,}\cr
(g_2\otimes g_1,b)&\text{if $b\notin W_E$,}
\end{cases}
%(g_{w(1)}\otimes g_{w(2)},w)\quad \textrm{for $g_1$, $g_2$ in $\GL(V)$ and $w\in \W_F$},
\end{equation}
where $g_1,g_2\in \GL_n(\C)$ and $a\in W_F$. If $\phi\in\Phi(M,F)$, then $r_\rmA\circ\phi\in\Phi(\GL_{n^2},F)$.

\begin {lem} \label{eqn:depth_u}
We have 
\[\dep(r_\rmA\circ\phi)=\dep(\phi),\] for any $\phi\in\Phi(M,F)$.
\end{lem}
\begin{proof} Let $\phi=(a_\phi,\nu)\in\Phi(M,F)$. From (\ref{eqn:Asai})  we get
\[(r_\rmA\circ\phi)(w)=\begin{cases}
(g_1(w)\otimes g_2(w),\nu(w))&\text{if $\nu(w)\in W_E$,}\cr
(g_2(w)\otimes g_1(w),\nu(w))&\text{if $\nu(w)\notin W_E$.}
\end{cases}
\]
Since $H^1(W_F/W_F^{r+},M^\vee)$ are the cohomology classes which can be represented by a
$1$-cocycle whose restriction to $W^{r+}$ is trivial,
we have $\alpha_\phi\in H^1(W_F/W_F^{r+},M^\vee)^{W_F^{r+}}$ if and only if 
\begin{equation} \label{eqn:triv}
\text{$g_1(w)\otimes g_2(w)=\rI_n\otimes \rI_n$ for every $w\in W_F^{r+}$,}
\end{equation}
where $\rI_n$ denotes the identity matrix in $\GL_n(\C)$.  

But (\ref{eqn:triv}) is satisfied if and only if we have $W_F^{r+}\subset \ker(r_\rmA\circ\phi)$. Then the result follows by Definition \ref{Oidef} .
\end{proof}

The {\em Asai lift} of $\pi_E\in\Pi(\GL_n,E)$ is the representation $\As(\pi_E)\in\Pi(\GL_{n^2},F)$ with $L$-parameter $r_\rmA\circ\phi$, where $\phi=(\lambda_M\circ\iota^{-1})(\pi_E)$, that is,
\begin{equation} \label{eqn:Asai2}
\As(\pi_E):=\lambda_{\tG}^{-1}(r_\rmA\circ\phi)=(\lambda_{\tG}^{-1}\circ r_\rmA\circ\lambda_M\circ\iota^{-1})(\pi_E)=
(\lambda_{\tG}^{-1}\circ r_\rmA\circ\Sh^{-1}\circ\lambda_G)(\pi_E). 
\end{equation} 

\begin{thm} \label{thm:depth_Asai}
%Let $\pi\in \Pi(M,F)$ and write $\phi:=\lambda_M(\pi)\in\Phi(M,F)$. 
Let  $\pi_E\in\Pi(\GL_n,E)$. We have  
\[\dep(\As(\pi_E))=\varphi_{E/F}(\dep(\pi_E)).\]
\end{thm}
\begin{proof}
Since $\lambda_{\tG}$ is depth preserving \cite[Theorem 2.9]{ABPS},  it follows from (\ref{eqn:Asai2}), that
\[\dep(\As(\pi_E))=\dep(r_\rmA\circ\Sh^{-1}\circ\lambda_G)(\pi_E).\]
By combining  Theorem~\ref{shr} and Lemma~\ref{eqn:depth_u}, we obtain
\[\dep(\As(\pi_E))=\dep((\Sh^{-1}\circ\lambda_G)(\pi_E))=\varphi_{E/F}(\dep(\lambda_G(\pi_E))).
\]
Since $\lambda_{G}$ is depth preserving \cite[Theorem 2.9]{ABPS}, we have
\[\dep(\As(\pi_E))=\varphi_{E/F}(\dep(\pi_E)).\]
\end{proof}

Let $\fo_E$ denote the ring of integers of $E$, let $\fp_E$ be the maximal ideal of $\fo_E$, and let $q_E$ be the order of $\fo_E/\fp_E$. Let $\psi_E$ be a continuous nontrivial additive character of $E$ and let $c(\psi_E)$ denote the largest integer $c$ such that $\psi$ is trivial on $\fp_E^{-c}$. 
Let  $\pi_E$ an essentially square-integrable irreducible representation of $\GL_{n},E)$. Its Godement-Jacquet local constant $\epsilon(s,\pi_E,\psi_E)$ takes the form
\[\epsilon(s,\pi_E,\psi_E)=\epsilon(0,\pi_E,\psi_E)\,\cdot\,q_E^{-f(\pi_E,\psi_E)s}, \]
where $s\in \C$ and $\epsilon(0,\pi_E,\psi_E)\in\C^\times$. The integer $f(\pi_E):=f(\pi_E,\psi_E)- n c(\psi_E)$ is called the {\em conductor} of $\pi_E$. We recall that $f(\pi_E)-n$ is the Swan conductor of $\pi_E$ (see for instance \cite[\S~4.3.2]{Bus}). We write (as in \cite[\S~5.3.2]{Bus}):
\begin{equation} \label{eqn:sw}
\varsigma(\pi_E):=\frac{f(\pi_E)-n}{n}.
\end{equation}

%We will next try to compare our result with Theorem 6.1 in \cite{AM1} (see also \cite[Theorem 6.1]{AM2} and Section 7 in \cite{AM2}).

\begin{cor} \label{cor:cond_Asai}
We assume $n\ge 2$. For any essentially square-integrable irreducible representation of $\GL_{n},E)$, we  have
\[\varsigma(\As(\pi_E))=\varphi_{E/F}(\varsigma(\pi_E)).
\]
\end{cor}
\begin{proof}
From \cite[Theorem 2.7]{ABPS}, we have (since $n\ge 2$)
\[f(\pi_E)=n\,\dep(\pi_E))+n.\]
It gives \[\varsigma(\pi_E)=\dep(\pi_E).\]
Similarly, we have
\[\varsigma(\As(\pi_E))=\dep(\As(\pi_E)).\]
Then the result follows from Theorem \ref{thm:depth_Asai}.
\end{proof}

\appendix 
\section{Moy--Prasad filtration of Weil--restricted groups \\[0.1cm] by Jessica Fintzen} \label{section-JF}
\markright{\textsc{Moy--Prasad filtration of Weil--restricted groups}}

Let $E/F$ be a finite Galois extension of non-archimedean local fields with ramification index $e$. Let $G$ be a connected reductive group defined over $E$ and set $M:= \fR_{E/F}\, G$  the Weil restriction of scalars of $G$ for the field extension $E/F$. We denote by $\iota: G({E}) \xrightarrow{\simeq} M({F})$ the isomorphism arising from the defining adjunction property of $M$. 

\markleft{\textsc{J. Fintzen}}

In this appendix we are going to prove (Proposition \ref{prop:depth-MP}) that for every $x \in \sB(G, E)$ and $r \in \bR_{\geq 0}$, we have 
\begin{equation}\label{eq:main} \iota(G(E)_{x,er})=M(F)_{i_{\sB}(x),r} , \end{equation}
where $i_\sB:\sB(G,E) \xrightarrow{\simeq} \sB(M,F)$ is an identification of the (reduced) Bruhat--Tits building $\sB(G,E)$ of $G$ over $E$ with the (reduced) Bruhat--Tits building  $\sB(M,F)$ of $M$ over $F$ that we will define in Definition \ref{def:isB}.

We are using the notation from the main part of the paper ``Comparison of the depths on both sides of the local Langlands correspondence for Weil-restricted groups'', i.e. if $H$ is a (connected) reductive group over a non-archimedean local field $K$,   $x$ a point in the (reduced) Bruhat--Tits building $\sB(H,K)$ of $H$ over $K$, and $r \in \bR_{\geq 0}$, then we denote by $H(K)_{x,r}$ the corresponding Moy--Prasad filtration subgroup (\cite{MoPr1, MoPr2}), and we write $H(K)_{x,r+}$ for the subgroup $\bigcup_{s>r} H(K)_{x,s}$ of $H(K)$.

We also fix a separable closure $F^{\sep}$ of $F$ and view all separable field extensions of $F$ inside $F^{\sep}$. For any finite separable extension $K$ of $F$, we write $K^{\ur}$ for its maximal unramified extension (contained in $F^{\sep}$).

In order to define and prove \eqref{eq:main}, we will first work over maximal unramified extensions and then combine the results with \'etale descent. We write $\Gur:=G_{E^{\ur}}$ and define $\Mur:=\fR_{E^{\ur}/F^{\ur}}\Gur$. Note that if $E/F$ is not totally ramified, then $\Mur \not\simeq M_{F^{\ur}}$.

For a torus $T$ defined over $F^{\ur}$, we denote by $T^{ft}$ the ft-N\'eron model of $T$ (\cite{CY}, see also \cite{B}) and by $T^{ft, 0}$ the connected component of $T^{ft}$ that contains the identity.
\begin{lem} \label{lem:neron}
	Let $T$ be a torus defined over $F^{\ur}$. Then we have $(\fR_{E^{\ur}/F^{\ur}} T)^{ft, 0} \simeq \fR_{E^{\ur}/F^{\ur}} (T^{ft, 0})$. 
\end{lem}
\begin{proof}
	By \cite[3.1.4~Satz]{B} we have $(\fR_{E^{\ur}/F^{\ur}} T)^{ft} \simeq \fR_{E^{\ur}/F^{\ur}} (T^{ft})$. Since $T^{ft}$ is smooth and affine, we have by \cite[Proposition~A.5.2(4)]{CGP} that $\fR_{E^{\ur}/F^{\ur}} (T^{ft, 0})$ is an open subgroup scheme of $\fR_{E^{\ur}/F^{\ur}} (T^{ft})\simeq(\fR_{E^{\ur}/F^{\ur}} T)^{ft}$, and by \cite[Proposition~A.5.11(3)]{CGP} the open subgroup scheme $\fR_{E^{\ur}/F^{\ur}} (T^{ft, 0})$ has geometrically connected fibers, hence it is the identity component of $(\fR_{E^{\ur}/F^{\ur}} T)^{ft}$. 
\end{proof}

Let $\TGur$ be a maximally split, maximal torus of $\Gur=G_{E^{\ur}}$ defined over $E^{\ur}$, and let $\TMur:=\fR_{E^{\ur}/F^{\ur}}\TGur$. Then $\TMur$ is a maximally split, maximal torus of $\Mur=\fR_{E^{\ur}/F^{\ur}}\Gur$, and by \cite[Proposition~A.5.15]{CGP} all maximally split, maximal tori of $\Mur$ arise in this way. 

Let $\SGur$ be the maximal split subtorus of $\TGur$ and $\SMur$ the maximal split subtorus of $\TMur$. Then $\SMur$ is contained in $\fR_{E^{\ur}/F^{\ur}}\SGur\subset \fR_{E^{\ur}/F^{\ur}}\TGur =\TMur$. We obtain a map 
$$ i_S: X^*(\SGur) \ra X^*(\SMur)$$
by sending $f \in \Hom_{E^{\ur}}(\SGur, \bbG_m)=X^*(\SGur) $ to the element $i_S(f) \in \Hom_{F^{\ur}}(\SMur, \bbG_m)= X^*(\SMur)$ whose composition with $\bbG_m \hookrightarrow \fR_{E^{\ur}/F^{\ur}}\bbG_m$ yields $\fR_{E^{\ur}/F^{\ur}}(f)|_{\SMur}$. Note that the map $i_S$ is an isomorphism. We use this isomorphism $i_S$ to identify $X^*(\SGur)$ and $X^*(\SMur)$. Under this identification the restricted root system $\Phi(\Gur,\SGur)$ of $\Gur$ with respect to $\SGur$ gets identified with the restricted root system $\Phi(\Mur,\SMur)$ of $\Mur$ with respect to $\SMur$ .

Let $a \in \Phi(\Gur,\SGur)=\Phi(\Mur,\SMur)$, and let $\UGur_a$ be the corresponding root subgroup of $\Gur$, i.e., the connected unipotent (closed) subgroup of $\Gur$ normalized by $\SGur$ whose Lie algebra is the sum of the root spaces corresponding to the roots that are a positive integral multiple of $a$. Similarly, we denote by $\UMur_a$ the root subgroup of $\Mur$  corresponding to $a$. Then we have
\begin{equation} \label{eqn:Us}
\UMur_a= \fR_{E^{\ur}/F^{\ur}}\UGur_a \subset  \fR_{E^{\ur}/F^{\ur}}\Gur  .
\end{equation}

Let $K$ be a finite Galois extension of $F^{\ur}$ containing $E^{\ur}$ and such that $\TGur \times_{E^{\ur}} {K}$ is split. We fix a Chevalley--Steinberg system 
$$\{x_{\alpha}^K:\bbG_a \ra \U_{\alpha}^K\}_{\alpha \in \Phi} $$
of $\Gur$ with respect to $\TGur$,  
where we write $\Phi := \Phi(\Gur_{K},\TGur \times_{E^{\ur}} {K})$ and $\U_{\alpha}^K$ denotes the root subgroup of $\Gur_K$ corresponding to $\alpha$, see \cite[\S2.1]{Fin} for the notion of a Chevalley--Steinberg system, which is based on \cite{BT2}. Recall that if we write  $K_\alpha$ for the fixed subfield of $K$ of the stabilizer $\Stab_{\Gal(K/E^{\ur})}(\alpha)$ of $\alpha$ in $\Gal(K/E^{\ur})$ (for  $\alpha \in \Phi$), then $x_{\alpha}^K$ is by definition of a Chevalley--Steinberg system defined over $K_\alpha$.

We will now show how this Chevalley--Steinberg system of $\Gur$ with respect to $\TGur$ yields a Chevalley--Steinberg system of $\Mur$ with respect to $\TMur$.

First, note that 
$$\Mur \times_{F^{\ur}} K \simeq \prod_{f \in \Hom_{F^{\ur}}(E^{\ur},K)} \Gur \times_{E^{\ur}, f} K, $$
which contains the split torus
$$\TMur \times_{F^{\ur}} K \simeq \prod_{f \in \Hom_{F^{\ur}}(E^{\ur},K)} \TGur \times_{E^{\ur}, f} K. $$
For later use, we fix for every $f \in \Hom_{F^{\ur}}(E^{\ur},K)$ an element $\wt f \in \Gal(K/F^{\ur})$ such that $\wt f|_{E^{\ur}}=f$. We write $id: E^{\ur} \hookrightarrow K$ for the inclusion of $E^{\ur}$ into $K$ arising from our convention to view both fields within the same fixed separable closure, and we choose $\wt{id}$ to be the identity element in $\Gal(K/F^{\ur})$.
Let $\alpha \in \Phi = \Phi(\Gur_{K},\TGur \times_{E^{\ur}} {K})$ and $f \in \Hom_{F^{\ur}}(E^{\ur},K)$. Then we write $\alpha_f$ for the root in $\Phi^{\Mur}:=\Phi(\Mur \times_{F^{\ur}} K,\TMur \times_{F^{\ur}} K )$ obtained by composing the projection 
$$\TMur \times_{F^{\ur}} K \simeq \prod_{f' \in \Hom_{F^{\ur}}(E^{\ur},K)} \TGur \times_{E^{\ur}, f'} K \twoheadrightarrow  \TGur \times_{E^{\ur}, f} K$$
that sends  $\TGur \times_{E^{\ur}, f'} K$ to the identity for $f'\neq f$, with the composition of the following $K$-group scheme homomorphisms
$$\TGur \times_{E^{\ur}, f} K \simeq \TGur_K \times_{K, \wt f} K \xrightarrow{\alpha \times id} \bbG_m \times_{K, \wt f} K \xrightarrow{\simeq} \bbG_m  .$$

Note that 
$$\Phi^{\Mur}=\Phi(\Mur \times_{F^{\ur}} K,\TMur \times_{F^{\ur}} K ) = \{\alpha_f \, | \, \alpha \in \Phi, f \in \Hom_{F^{\ur}}(E^{\ur},K)\}.$$

For $f \in \Hom_{F^{\ur}}(E^{\ur},K)$, we write 
$$i_f : \Gur \times_{E^{\ur}, f} K \hookrightarrow \prod_{f' \in \Hom_{F^{\ur}}(E^{\ur},K)} \Gur \times_{E^{\ur}, f'} K \simeq  \Mur_K $$
for the inclusion whose image is the factor corresponding to $f$, and we define the $K$-group scheme homomorphism
\begin{eqnarray*}
	x_{\alpha_f}^K&:& \bbG_a \simeq \bbG_a \times_{K, \wt f} K \xrightarrow{x_\alpha \times id} \U_\alpha^K \times_{K, \wt f} K \subset  \Gur \times_{E^{\ur}, f} K \\
	& & \xrightarrow{i_f} \prod_{\Hom_{F^{\ur}}(E^{\ur},K)} \Gur \times_{E^{\ur}, f'} K \simeq  \Mur_K .
\end{eqnarray*}

Note that the image of $ \U_\alpha^K \times_{K, \wt f} K $ via $i_f$ in $\Mur_K$ is 
the root subgroup  $\U_{\alpha_f}^K$ of $\Mur_K$ attached to the root $\alpha_f$. Thus $x_{\alpha_f}^K$ factors through the root subgroup $\U_{\alpha_f}^K$. 
\begin{lem}\label{lem:Chevalley-Steinberg-Mur}
	The set $\{x_{\alpha_f}^K : \bbG_a \ra \U_{\alpha_f}^K \}_{\alpha_f \in \Phi^{\Mur}}$ forms a Chevalley--Steinberg system of $\Mur$ with respect to $\TMur$.
\end{lem}
\begin{proof}
	Let $\alpha_f \in \Phi^{\Mur}$, i.e. $\alpha \in \Phi$ and $f \in \Hom_{F^{\ur}}(E^{\ur},K)$.  For  $\gamma \in \Gal(K/F^{\ur})$, we can write $\gamma \wt f =\wt f' \gamma_0 $ for some $f' \in \Hom_{F^{\ur}}(E^{\ur},K)$ and $\gamma_0 \in \Gal(K/E^{\ur})$. Then we have $\gamma(\alpha_f)=(\gamma_0(\alpha))_{f'}$. Hence the fixed field $K_{\alpha_f}$ of the stabilizer $\Stab_{\Gal(K/F^{\ur})}(\alpha_f)$ of $\alpha_f$ is $\wt f K_\alpha \wt f^{-1}$. Since $x_\alpha$ is defined over $K_\alpha$, we deduce from the  construction of $x_{\alpha_f}^K$ that $x_{\alpha_f}^K$ is defined over $K_{\alpha_f}$. 
	We distinguish two cases. 
	
	Case 1: The restriction of $\alpha_f$ to $\Phi(\Mur, \SMur)$ is not divisible. In this case it remains to check that for all  $\gamma \in \Gal(K/F^{\ur})$, we have $x_{\gamma(\alpha_f)}^K=\gamma \circ x_{\alpha_f}^K \circ \gamma^{-1}$. Write $\gamma =\wt f' \gamma_0 \circ \wt f^{-1}$ for some $f' \in \Hom_{F^{\ur}}(E^{\ur},K)$ and $\gamma_0 \in \Gal(K/E^{\ur})$. Note that the restriction of $\alpha$ to $\Phi(\Gur, \SGur)$  agrees with the restriction of $\alpha_f$ to $\Phi(\Mur, \SMur)$ under the above identification of $\Phi(\Gur, \SGur)$ with $\Phi(\Mur, \SMur)$. Thus the restriction of $\alpha$ to $\Phi(\Gur, \SGur)$ is non-divisible, and since $\{x_\alpha^K\}_{\alpha \in \Phi}$ form a Chevalley--Steinberg system, we have  $x_{\gamma_0(\alpha)}^K=\gamma_0 \circ x_\alpha^K \circ \gamma_0^{-1}$. Thus we obtain
	$$ \gamma \circ x_{\alpha_f}^K \circ \gamma^{-1} = \wt f' \circ \gamma_0 \circ x_{\alpha_{id}}^K \circ \gamma_0^{-1} \wt (f')^{-1} =  \wt f' \circ x_{\gamma_0(\alpha)_{id}}^K \circ (f')^{-1}= x_{\gamma_0(\alpha)_{f'}}^K=x_{\gamma(\alpha_f)}^K.$$
	
	Case 2: The restriction of $\alpha_f$ to $\Phi(\Mur, \SMur)$ is divisible. Hence the restriction of $\alpha$ to $\Phi(\Gur, \SGur)$ is divisible and there exist $\beta$ and $\beta'$ with $\alpha=\beta+\beta'$, $\beta|_{\SGur}=\beta'|_{\SGur}$, and $K_\beta=K_{\beta'}$ is a quadratic extension of $K_\alpha$. Hence $\alpha_f=\beta_f+\beta'_f$ and $K_{\beta_f}=\wt fK_{\beta}\wt f^{-1}$ is a quadratic extension of $K_{\alpha_f}=\wt fK_{\alpha}\wt f^{-1}$. It remains to show that for $\gamma \in \Gal(K/K_{\alpha_f})$, we have 
	\begin{equation}\label{eqn-Chevalley-Steinberg-iii} x_{\gamma(\alpha_f)}^K=\gamma \circ x_{{\alpha_f}}^K \circ \gamma^{-1} \circ \epsilon ,
	\end{equation}
	where $\epsilon \in \{\pm 1\}$ is 1 if and only if $\gamma$ induces the identity on $K_{{\beta_f}}$. Note that if we write $\gamma=\wt f \gamma_0 \wt f^{-1}$ with $\gamma_0 \in \Gal(K/K_\alpha)$, then $\gamma$ induces the identity on $K_{{\beta_f}}$ if and only if $\gamma_0$ induces the identity on $K_{{\beta}}$. Hence the desired identify \eqref{eqn-Chevalley-Steinberg-iii} follows from the property $x_{\gamma_0(\alpha)}^K=\gamma_0 \circ x_{{\alpha}}^K \circ \gamma_0^{-1} \circ \epsilon$ of the Chevalley--Steinberg system $\{x_\alpha^K\}_{\alpha \in \Phi}$. 
\end{proof}

Recall that following \cite{BT2} we obtain a parametrization of root groups from our Chevalley--Steinberg systems. More precisely, let $a \in \Phi(\Gur,\SGur)$ and fix $\alpha \in \Phi$  such that $\alpha|_{\SGur}=a$. Recall that $x_\alpha$ is defined over $K_\alpha$ by the properties of a Chevalley--Steinberg system.
If $a$ is not multipliable, then 
$$x_a:= \fR_{K_\alpha/E^{\ur}} x_\alpha^K:\fR_{K_\alpha/E^{\ur}} \bbG_a \xrightarrow{\simeq} \UGur_a$$ 
is the parametrization of $\UGur_a$ corresponding to the Chevalley--Steinberg system.    
If $a$ is multipliable, then let $\wt \alpha \in \Phi$ such that $\wt \alpha|_{\SGur}=a$ and $\alpha+\wt \alpha \in \Phi$. Using $x_\alpha^F, x_{\wt \alpha}^F$ and $x_{\alpha+\wt \alpha}^F$, following \cite[4.1.9]{BT2} (see also \cite[Section~2.2]{Fin} for an exposition) we obtain a parametrization 
$$x_a:\fR_{K_{\alpha+\wt \alpha}/E^{\ur}} H_0(K_\alpha,K_{\alpha+\wt \alpha}) \xrightarrow{\simeq} \UGur_a$$
of $\UGur_a$, where $H_0(K_\alpha,K_{\alpha+\wt \alpha})$ is as defined in \cite[4.1.9]{BT2} (see also \cite[Section~2.2]{Fin}). Composing the inverse of $x_a$ with the valuation on $(\fR_{K_\alpha/E^{\ur}} \bG_a) (E^{\ur})=K_\alpha$ or with a scaling of the valuation on the second factor of $K_\alpha \times K_\alpha \supset (H_0(K_\alpha,K_{\alpha+\wt \alpha}))(K_{\alpha+\wt \alpha})$ $=(\fR_{K_{\alpha+\wt \alpha}/E^{\ur}} H_0(K_\alpha,K_{\alpha+\wt \alpha}))(E^{\ur}) $ as described in \cite[4.2.2]{BT2} (see also \cite[Section~2.2]{Fin}), we obtain a valuation
$$ \varphi_a: \UGur_a(E^{\ur}) \ra \frac{1}{2[K_\alpha:E^{\ur}]}\bZ \cup \{\infty\}$$
of $\UGur_a(E^{\ur})$.
These valuations $\{\varphi_a\}_{a \in \Phi(\Gur,\SGur)}$ determine a point $x_\varphi$ in the apartment $\sA(\SGur, E^{\ur})$ corresponding to $\SGur$, and all other points in the apartment correspond to valuations of the form $(\wt \varphi_a:\UGur_a(E^{\ur}) \ra \bR \cup \{ \infty \})_{a \in \Phi(\Gur,\SGur)}$ with $\wt \varphi_a(u)=\varphi_a(u)+a(v)$ for some $v \in X_*(\SGur)\otimes \bR$ and for all $u \in \UGur_a(E^{\ur})$, $a \in  \Phi(\Gur,\SGur)$ . 

Similarly, the Chevalley--Steinberg system $\{x_{\alpha_f}^K\}_{\alpha_f \in \Phi^{\Mur}}$ yields valuations $\{\varphi_a^{\Mur}\}_{a \in \Phi(\Mur,\SMur)}$ that determine a point $x_{\varphi^{\Mur}}$ in the apartment $\sA(\SMur, F^{\ur})$ corresponding to $\SMur$, and all other points in the apartment correspond to valuations of the form $(\wt \varphi^{\Mur}_a:\UMur_a(F^{\ur}) \ra \bR \cup \{ \infty \})_{a \in \Phi(\Mur,\SMur)}$ with $\wt \varphi^{\Mur}_a(u)=\varphi^{\Mur}_a(u)+a(v)$ for some $v \in X_*(\SMur)\otimes \bR$ and for all $u \in \UMur_a(F^{\ur})$, $a \in  \Phi(\Mur,\SMur)$. 	

Using the identification of $X^*(\SGur)$ with $X^*(\SMur)$ via $i_S$ and the resulting identification of $X_*(\SGur)$ with $X_*(\SMur)$, we obtain a bijection 
$$i_{\sA}: \sA(\SGur, E^{\ur}) \xrightarrow{\simeq} \sA(\SMur, F^{\ur})$$
by sending $\varphi_a + a(v)$ to $\varphi^{\Mur}_a +\frac{1}{e} \cdot a(v)$, where $e=[E^{\ur}:F^{\ur}]$.

Let $\iota^{\ur}: \Gur({E^{\ur}}) \xrightarrow{\simeq} \Mur({F^{\ur}})$ denote the isomorphism arising from the defining adjunction property of $\Mur$.
\begin{lem} \label{lem:depth-preservation}
	Let $x \in \sA(\SGur, E^{\ur})$ and $r \in \bR_{\geq 0}$.
	Then $$ \iota^{\ur}(\Gur(E^{\ur})_{x,er})=\Mur(F^{\ur})_{i_{\sA}(x),r} .$$
\end{lem}

\begin{proof}
	Let $a \in \Phi(\Mur, \SMur)$, and let $\xMur_a$ denote the parametrization of $\UMur_a$ attached to the Chevalley--Steinberg system $\{x_{\alpha_f}^K\}_{\alpha_f \in \Phi^{\Mur}}$. If $a$ is non-multipliable, let $\alpha_{id} \in \Phi^{\Mur}$ such that $\alpha_{id}|_{\SMur}=a$. Then
	$$ x^{\Mur}_{a}= \fR_{K_{\alpha_{id}}/F^{\ur}} x_{\alpha_{id}}^K = \fR_{K_{\alpha}/F^{\ur}} x_{\alpha_{id}}^K =  \fR_{E^{\ur}/F^{\ur}} (\fR_{K_{\alpha}/E^{\ur}} x_{\alpha_{id}}^K) $$
	and using \eqref{eqn:Us} and the definition of $x_{\alpha_{id}}^K$ we obtain that 
	\begin{equation}\label{eqn:xMur}
	x^{\Mur}_a=\fR_{E^{\ur}/F^{\ur}} x_a .
	\end{equation}
	Similarly we observe that equation \eqref{eqn:xMur} also holds if $a$ is non-multipliable.
	Hence 
	$$ \varphi^{\Mur}_a \circ \iota^{\ur}|_{\UGur_a(E^{\ur})}=\frac{1}{e} \cdot \varphi_a .$$
	This implies that the bijection $i_\sA$ of apartments induces a bijection $e \cdot i^*_\sA$ between the set of affine roots $\Psi^{\Mur}_{F^{\ur}}$ of $\sA(\SMur, F^{\ur})$ and the affine roots $\Psi^{\Gur}_{E^{\ur}}$ of $\sA(\SGur, E^{\ur})$. Hence we have 
	$$ \iota^{\ur}\left(\< \UGur_\psi , |  \, \psi \in \Psi^{\Gur}_{E^{\ur}}, \psi(x) \geq er \>\right) = \< \UMur_\psi \, | \, \psi \in \Psi^{\Mur}_{F^{\ur}}, \psi(i_\sA(x)) \geq r \> ,$$
	where $\UGur_\psi=\{u \in \UGur_{\dot \psi}(E^{\ur}) \, | \, \varphi_{\dot \psi}(u) \geq \psi(x_\varphi)\}$ with $\dot \psi$ denoting the gradient of $\psi$, and similarly for $\UMur_\psi$.
	
	By Lemma \ref{lem:neron} we have $\iota^{\ur}(\TGur_0)=\TMur_0$, and hence $\iota^{\ur}(\TGur_{er})=\TMur_r$ for $r \in \bR_{\geq 0}$. Thus we obtain 
	\begin{eqnarray*}
		\iota^{\ur}(\Gur(E^{\ur})_{x,er})&=&\iota^{\ur}\left(\<\TGur_{er},  \UGur_\psi , |  \, \psi \in \Psi^{\Gur}_{E^{\ur}}, \psi(x) \geq er  \>\right) \\
		&=& \< \TMur_r, \UMur_\psi \, | \, \psi \in \Psi^{\Mur}_{F^{\ur}}, \psi(i_\sA(x)) \geq r \> =\Mur(F^{\ur})_{i_{\sA}(x),r} 
	\end{eqnarray*}
\end{proof}

\begin{cor} \label{corollary:depth-preservation}
	The bijection $i_\sA$ extends to a bijection $i^{\ur}_\sB: \sB(\Gur, E^{\ur}) \xrightarrow{\simeq} \sB(\Mur, F^{\ur})$ that is compatible with the action of $ \Gur({E^{\ur}}) \xrightarrow{\iota^{\ur} \, \simeq} \Mur({F^{\ur}})$ and such that for $x \in \sB(\Gur, E^{\ur})$ and $r \in \bR_{\geq 0}$ we have $$ \iota^{\ur}(\Gur(E^{\ur})_{x,er})=\Mur(F^{\ur})_{i^{\ur}_\sB(x),r} .$$
\end{cor}
It was pointed out to us that the isomorphisms between the buildings $\sB(\Gur, E^{\ur})$ and $\sB(\Mur, F^{\ur})$ has already been observed by \cite[Proposition~4.6]{HR} without the statement about the comparison of the Moy--Prasad filtration subgroups. 
\begin{proof}[Proof of Corollary \ref{corollary:depth-preservation}]
	We have a bijection $\iota^{\ur} \times i_\sA: \Gur(E^{\ur}) \times \sA(\SGur, E^{\ur}) \ra \Mur(F^{\ur}) \times \sA(\SMur, F^{\ur})$ and we will show that it descends to a bijection $i^{\ur}_\sB:  \sB(\Gur, E^{\ur}) \xrightarrow{\simeq} \sB(\Mur, F^{\ur})$. Recall that the equivalence relation on $\Gur(E^{\ur}) \times \sA(\SGur, E^{\ur})$ that defines $\sB(\Gur, E^{\ur})$ is given by $(g_1,x_1) \sim (g_2,x_2)$ if and only if there exists $n \in N_{\Gur}(\SGur)(E^{\ur})$ such that $x_2=n.x_1$ and $g_1^{-1}g_2n \in \Gur(E^{\ur})_{x_1,0}$, where $N_{\Gur}(\SGur)$ denotes the normalizer of  $\SGur$ in $\Gur$. We have analogous relations defining $ \sB(\Mur, F^{\ur})$. Note that $i^{\ur}(N_{\Gur}(\SGur)(E^{\ur}))=N_{\Mur}(\SMur)(F^{\ur})$ and $\iota^{\ur}(\Gur(E^{\ur})_{x,0})=\Mur(F^{\ur})_{i_{\sA}(x),0}$ for $x \in \sA(\SGur, E^{\ur})$ by Lemma \ref{lem:depth-preservation}. Thus $i_\sA$ is equivariant under the action of $(N_{\Gur}(\SGur)(E^{\ur})) \xrightarrow{\simeq} N_{\Mur}(\SMur)(F^{\ur})$, and the equivalence relation on  $ \Gur(E^{\ur}) \times \sA(\SGur, E^{\ur})$ defining $\sB(\Gur, E^{\ur})$ corresponds under $\iota^{\ur} \times i_\sA$ to the equivalence relation on $ \Mur(F^{\ur}) \times \sA(\SMur, F^{\ur})$ defining $\sB(\Mur, F^{\ur})$. The corollary follows.
\end{proof}

This concludes our study of the Moy--Prasad filtration subgroups over maximal unramified extensions. We will now employ \'etale descent to obtain the desired results over our local fields $E$ and $F$.
We write $E_{\ur}$ for the maximal unramified extension of $F$ contained in $E$. Then we have $E \otimes_{E_{\ur}} F^{\ur} = E^{\ur}$ and every $ f \in \Hom_F(E_{\ur}, F^{\ur})$ yields an element of $\Hom_F(E, E^{\ur})$ that we also denote by $f$. Thus we obtain
\begin{equation}\label{eqn:M-M} M \times_F F^{\ur} = \fR_{E/F} G \times_F F^{\ur} \simeq \prod_{f \in \Hom_F(E_{\ur},F^{\ur})} \fR_{E^{\ur}/F^{\ur}} (G \times_{E, f} E^{\ur}). \end{equation}
Hence 
$$ \sB(M, F)=\sB(M \times_F F^{\ur},F)^{\Gal(F^{\ur}/F)}$$
with $$\sB(M \times_F F^{\ur},F) = \prod_{f \in \Hom_F(E_{\ur},F^{\ur})} \sB(\fR_{E^{\ur}/F^{\ur}} (G \times_{E, f} E^{\ur}), F^{\ur}). $$
By composing the latter product with the projection onto the factor corresponding to $f=1$, we obtain a bijection
\begin{eqnarray*}
	i_{\sB,M,\Mur}: \sB(M,F) & \xrightarrow{\simeq} & \sB(\fR_{E^{\ur}/F^{\ur}} (G \times_E E^{\ur}), F^{\ur})^{\Gal(F^{\ur}/E_{\ur})} = \sB(\Mur, F^{\ur})^{\Gal(F^{\ur}/E_{\ur})} .
\end{eqnarray*}
Similarly, composing Equation \eqref{eqn:M-M} with the projection onto the factor corresponding to $f=1$, we obtain
an isomorphism 
\begin{eqnarray*}
	\iota_{M,\Mur}: M(F) \xrightarrow{\simeq} \Mur(F^{\ur})^{\Gal(F^{\ur}/E_{\ur})} 
\end{eqnarray*}
such that for $x \in \sB(M,F)$ and $r \in \bR_{\geq 0}$ we have
\begin{equation} \label{eqn:M-Mx}
M(F)_{x,r}=((M \times_F F^{\ur})(F^{\ur})_{x,r})^{\Gal(F^{\ur}/F)}=\iota_{M,\Mur}^{-1}\left((\Mur(F^{\ur})_{i_{\sB,M,\Mur}(x),r})^{\Gal(F^{\ur}/E_{\ur})}\right).
\end{equation}

\begin{defn} \label{def:isB}
	We let 
	$$i_\sB: \sB(G,E) \xrightarrow{\simeq} \sB(M,F)$$ denote the bijection obtained as the
	composition of the restriction of $i^{\ur}_\sB$ to  $\sB(G,E) $:
	$$i^{\ur}_\sB: \sB(G,E) = \sB(\Gur, E^{\ur})^{\Gal(E^{\ur}/E)}  \xrightarrow{\simeq} \sB(\Mur, F^{\ur})^{\Gal(F^{\ur}/E_{\ur})} $$
	with $i_{\sB,M,\Mur}^{-1}$ .
\end{defn}
Recall that $\iota: G({E}) \xrightarrow{\simeq} M({F})$ denotes the isomorphism arising from the defining adjunction property of $M=\fR_{E/F}\, G$. Then we obtain the following result.
\begin{prop} \label{prop:depth-MP}
	Let $x \in \sB(G, E)$ and $r \in \bR_{\geq 0}$. Then
	$$ \iota(G(E)_{x,er})=M(F)_{i_{\sB}(x),r} .$$
\end{prop}
\begin{proof}
	Combining Corollary \ref{corollary:depth-preservation} and Equation \eqref{eqn:M-Mx} we obtain
	\begin{eqnarray*}
		\iota(G(E)_{x,er}) & = & \iota\left((\Gur(E^{\ur})_{x,er})^{\Gal(E^{\ur}/E)}\right) = \iota_{M,\Mur}^{-1}\iota^{\ur}\left((\Gur(E^{\ur})_{x,er})^{\Gal(E^{\ur}/E)}\right)		\\
		& = & \iota_{M,\Mur}^{-1}\left((\Mur(F^{\ur})_{i^{\ur}_\sB(x),r})^{\Gal(F^{\ur}/E_{\ur})}\right)=    M(F)_{i_{\sB,M,\Mur}^{-1}(i^{\ur}_\sB(x)),r}  \\
		& = & M(F)_{i_{\sB}(x),r}
	\end{eqnarray*} 
\end{proof}

As an immediate corollary we deduce the following result.
\begin{cor} \label{ethm} Let $(\pi,V_\pi)$ be an irreducible smooth complex representation of $M(F)$. Then
\begin{eqnarray}\label{iota}
\dep(\iota^* \pi) = e(\dep (\pi)),
\end{eqnarray}
where $\iota^* \pi$ denotes the composition of $\iota$ with $\pi$ and $\dep(\cdot)$ denotes the depth of the corresponding representation.
\end{cor}

\begin{proof} This follows from Proposition \ref{prop:depth-MP} and the fact that $i_\sB$ is a bijection between $\sB(G,E)$ and $\sB(M,F)$.
\end{proof}

 \begin{comment} 
\textsc{Application to induced tori}

\begin{cor}   Let $T= \fR_{E/F} \mathbb{G}_m$ be an induced
torus, where $E/F$ is a finite and Galois extension of non-archimedean local fields. In the $\LLC$ for $T$, depth is preserved (for positive depth characters)  if and only if $E/F$ is tamely ramified.   
 
 If $E/F$ is wildly ramified then, for each  character $\chi$ of $T(F)$ such that $\dep(\chi)>0$,
 we have
\[
\dep(\lambda_T(\chi) ) > \dep(\chi), 
\]
and 
\[
\dep(\lambda_T(\chi)) / \dep(\chi) \to 1 \quad \textrm{as} \quad \dep(\chi) \to \infty.   
\]  
\end{cor}
\begin{proof}
We take $G = \GL_1$ in Theorem~\ref{main}.  Then $M=\fR_{E/F} G$ is an induced torus $T$. Note that $G^\vee = \C^\times$, and $T^\vee = (\C^\times)^d$.
The existence of 
\[
\lambda_T \colon \Pi(T(F), \C^\times) \simeq \Phi(T(F))
\]
is known, since the $\LLC$ is proved for algebraic tori (see \cite{La}). 
Then the statement follows swiftly from the above results.     
\end{proof}
\end{comment}

\markright{\textsc{References}}
\markleft{\textsc{References}}

\end{document}